\documentclass[a4,10pt]{article}
\usepackage{amssymb,latexsym,amsmath,amsthm,subfigure,graphicx,anysize,tikz,pgfplots,siunitx}
\usepackage[affil-it]{authblk}
\marginsize{3cm}{3cm}{3cm}{3cm}

\newcommand{\kb}{k_\mathrm{b}}
\newcommand{\eps}{\varepsilon}
\newcommand{\epsb}{\varepsilon_\mathrm{b}}
\newcommand{\mub}{\mu_\mathrm{b}}
\newcommand{\sigmab}{\sigma_\mathrm{b}}
\newcommand{\epsa}{\varepsilon_\mathrm{a}}
\newcommand{\sigmaa}{\sigma_\mathrm{a}}
\newcommand{\set}[1]{\left\{#1\right\}}
\newcommand{\abs}[1]{\left|#1\right|}
\newcommand{\E}{\mathrm{E}_{\inc}^{(z)}}
\newcommand{\mE}{\mathbf{E}}
\newcommand{\mH}{\mathbf{H}}

\newcommand{\mU}{\mathbf{U}}
\newcommand{\mV}{\mathbf{V}}
\newcommand{\mW}{\mathbf{W}}
\newcommand{\p}{\partial}
\newcommand{\rd}{\mathrm{d}}
\newcommand{\ma}{\mathbf{a}}

\newcommand{\ms}{\mathbf{z}}
\newcommand{\mx}{\mathbf{x}}
\newcommand{\vt}{\boldsymbol{\theta}}

\DeclareMathOperator*{\inc}{inc}
\DeclareMathOperator*{\tot}{tot}
\DeclareMathOperator*{\inp}{in}
\DeclareMathOperator*{\out}{out}
\DeclareMathOperator*{\meas}{meas}
\DeclareMathOperator*{\diam}{diam}
\DeclareMathOperator*{\area}{area}

\theoremstyle{plain}
\newtheorem{theorem}{Theorem}[section]

\newtheorem{corollary}{Corollary}[section]

\theoremstyle{remark}
\newtheorem{remark}{Remark}[section]
\newtheorem{example}{Example}[section]

\begin{document}

\title{On the application of subspace migration from scattering matrix with constant-valued diagonal elements in microwave imaging}
\author{Won-Kwang Park}
\affil{Department of Information Security, Cryptology, and Mathematics, Kookmin University, Korea.\\
e-mail: parkwk@kookmin.ac.kr}
\date{}

\maketitle

\begin{abstract}
We consider the application of a subspace migration (SM) algorithm to quickly identify small objects in microwave imaging. In various problems, it is easy to measure the diagonal elements of the scattering matrix if the location of the transmitter and the receiver is the same. To address this issue, several studies have been conducted by setting the diagonal elements to zero. In this paper, we generalize the imaging problem by setting diagonal elements of the scattering matrix as a constant with the application of SM. To show the applicability of SM and its dependence on the constant, we show that the imaging function of SM can be represented in terms of an infinite series of the Bessel functions of integer order, antenna number and arrangement, and applied constant. This result enables us to discover some further properties, including the unique determination of objects. We also demonstrated simulation results with synthetic data to support the theoretical result.
\end{abstract}

\section{Introduction}\label{sec:1}
There are considerable and intriguing inverse problems related to retrieving inhomogeneities embedded in a medium under multistatic measurement circumstances. This is an old but interesting problem due to its diverse applications, which significantly impact human life. For instance, it can be used in biomedical imaging \cite{A1,CJGT}, thermal therapy monitoring \cite{HSM2,KLKJS}, landmine detection \cite{CGSMMRW,GCGGC}, damage detection \cite{KJFF,SH}, and radar imaging \cite{C7,SK}. To this end, several researchers have explored and successfully applied various iterative (quantitative) and non-iterative (qualitative) reconstruction algorithms to address the different inverse scattering and microwave imaging problems.

Examples of iterative schemes include the Gauss-Newton method for biomedical imaging \cite{RMMP}, Born iterative method for brain stroke detection \cite{IBA} and compressive sensing imaging \cite{OAM}, Levenverg-Marquardt method for reconstructing parameter distribution \cite{FC}, and level-set method for shape reconstruction \cite{DL}. For a successful application of iterative-based algorithms, the iteration procedure must start with a good initial guess, which is close to the true solution. If not, it will be challenging to obtain good results due to the non-convergence issue, local minimizer problem, and requirement of high computational costs (see \cite{KSY,PL4} for instance).

Although complete information on parameter distribution or unknown targets cannot be retrieved through the non-iterative scheme, an outline shape of the target can be obtained quickly and adopted as a good initial guess. For example, bistatic method for object detection in microwave imaging \cite{SP2} and damage detection of civil structure \cite{KJFF}, MUltiple SIgnal Classification (MUSIC) algorithm for detecting internal corrosion \cite{AKKLV} and object detection in microwave imaging \cite{P-MUSIC6}, direct sampling method for diffusive optical tomography \cite{CILZ} and electrical impedance tomography \cite{CIZ}, topological derivative strategy for ultrasonic non-destructive testing \cite{AD} and detection of impedance obstacles \cite{LR2}, orthogonality sampling method for object detection in microwave imaging \cite{P-OSM2} and 3D objects \cite{LNST}, factorization method for breast cancer imaging \cite{CBC} and electrical impedance tomography \cite{HSW}, and linear sampling method for underwater acoustic imaging \cite{A7} and microwave imaging \cite{AHAM}.

In this paper, we explore an application of the SM for the rapid identification of a two-dimensional small object with different dielectric permittivity and electric conductivity values compared to the background homogeneous medium. SM is a recently developed and promising non-iterative imaging technique in inverse scattering problems and microwave imaging. It stands out for its speed, effectiveness in recognizing target locations and shapes, and robustness against random noise. As a result, SM has been successfully applied to various problems, including localization of small and extended targets \cite{AGKLS,BPV,P-SUB18}, crack imaging \cite{AGKPS,P-SUB3,P-SUB17}, and object detection in microwave imaging \cite{P-SUB10,P-SUB11,P-SUB16}.

In designing an imaging function of SM, it is essential to consider the complete elements of the so-called multistatic response matrix in inverse scattering problems and scattering matrix in microwave imaging. However, in some microwave imaging applications, collecting the diagonal elements of the scattering matrix becomes challenging due to the difficulty in measuring scattering parameter data  (elements of the scattering matrix) or discerning weak scattered signal from the relatively high antenna reflections when the transmitting and receiving antennas coincide. For a detailed discussion on this topic, refer to \cite{KLKJS,P-SUB16,SKLKLJC,SLP,SSKLJ}. Therefore, applying SM directly without unknown diagonal elements of the scattering matrix is highly challenging. SM has been applied by setting the unknown diagonal elements to zero. Numerical simulations demonstrate that setting the diagonal elements to a nonzero constant yields a poor result. As a result, most studies on microwave imaging problems have adopted the approach of setting the diagonal elements to zero. However, the theoretical rationale behind the effect of converting these elements to a constant value remains unexplored.

In this study, we consider the application of SM for the fast identification of small objects from a scattering matrix whose elements are scattering parameter data. By converting the unknown diagonal elements of the scattering matrix into a fixed constant, we introduce the imaging function of SM and show that it can be represented by an infinite series of the Bessel function of integer order, total number of antennas and their arrangement, applied frequency, material properties of the background, and applied constant. Based on the explored structure of the imaging function, we confirm that the imaging performance significantly depends on the applied constant and the object can be identified uniquely if the absolute value of such constant is zero or close to zero. We present simulation results from synthetic data to demonstrate the theoretical result.

The remainder of the paper is organized as follows. In Section \ref{sec:2}, we briefly present the concept of scattering parameter in the presence of small object and introduce the imaging function of SM from a scattering matrix whose diagonal elements are fixed constant. In Section \ref{sec:3}, we construct a mathematical theory for the imaging function, discuss various properties of the imaging function, including ideal and practical conditions for obtaining good results, and explain the unique determination of object. Section \ref{sec:4} presents a set of numerical simulation results for single and multiple objects to support the theoretical result. Finally, a short conclusion is presented in Section \ref{sec:5}.

\section{Scattering Parameters and Imaging Function of the SM}\label{sec:2}
This section briefly introduces the scattering parameters and the imaging function of SM for identifying small object $D$ with smooth boundary $\p D$. Throughout this paper, we denote $\Omega$ as a homogeneous region filled by a matching liquid and $\Lambda_n$ as a dipole antenna located at $\ma_n$ with $|\ma_n|=R$, $n=1,2,\cdots,N$. We assume that $D$ and $\Omega$ are characterized by their dielectric permittivity and electrical conductivity at a given angular frequency $\omega=2\pi f$, where $f$ denotes the ordinary frequency measured in \texttt{hertz}, and time-harmonic dependence $e^{-i\omega t}$. Thus, the value of the magnetic permeability is set to constant $\mu(\ms)\equiv\mub=4\pi\times \SI{e-7}{\henry/\meter}$ for every $\ms\in\Omega$. The values of permittivity of $\Omega$ and $D$ are denoted by $\epsb$ and $\epsa$, respectively. Analogously, $\sigmab$ and $\sigmaa$ denote the values of the conductivity of $\Omega$ and $D$, respectively. With this, we denote $k_0$ and $\kb$ as the lossless background and background wavenumbers that satisfy
\[k_0^2=\omega^2\epsb\mub\quad\text{and}\quad \kb^2=\omega^2\mub\left(\epsb-i\frac{\sigmab}{\omega}\right),\]
respectively.

Let $\mathcal{S}(n,m)$ be the scattering parameter, which is defined as the following ratio
\[\mathcal{S}(n,m)=\frac{\mathcal{V}_{\out}^{(n)}}{\mathcal{V}_{\inp}^{(m)}},\]
where $\mathcal{V}_{\inp}^{(m)}$ and $\mathcal{V}_{\out}^{(n)}$ denote the input voltages (or incident waves) at $\Lambda_{m}$ and the output voltages (or reflected waves) at $\Lambda_{n}$, respectively. We denote $\mathcal{S}_{\meas}(n,m)$ as the measurement data obtained by subtracting scattering parameters with and without $D$. Then, based on \cite{HSM2}, $\mathcal{S}_{\meas}(n,m)$ is given by
\begin{equation}\label{ScatteringParameter}
  \mathcal{S}_{\meas}(n,m)=-\frac{ik_0^2}{4\omega\mub}\int_{D}\left(\frac{\epsa-\epsb}{\epsb}+i\frac{\sigmaa-\sigmab}{\omega\epsb}\right)\mE_{\inc}(\ma_{m},\mx)\cdot\mE_{\tot}(\mx,\ma_{n})\rd\mx,
\end{equation}
where $\mE_{\inc}(\ma_{m},\mx)$ is the incident field due to the point current density at $\Lambda_{m}$ that satisfies
\[\nabla\times\mE_{\inc}(\ma_{m},\mx)=i\omega\mub\mH(\ma_{m},\mx)\quad\mbox{and}\quad\nabla\times\mH(\ma_{m},\mx)=(\sigma_\mathrm{b}-i\omega\eps_\mathrm{b})\mE_{\mbox{\tiny inc}}(\ma_{m},\mx),\]
and $\mE_{\tot}(\mx,\ma_{n})$ be the total field owing to the presence of the $D$ measured at $\Lambda_n$ that satisfies
\[\nabla\times\mE_{\tot}(\mx,\ma_{n})=i\omega\mub\mH(\mx,\ma_{n})\quad\mbox{and}\quad\nabla\times\mH(\mx,\ma_{n})=(\sigma(\mx)-i\omega\eps(\mx))\mE_{\tot}(\mx,\ma_{n}),\]
with transmission conditions at the boundary $\p D$. Here, $\mH$ denotes the magnetic field.

Now, assume that the following relations hold
\[\omega\epsb\gg\sigmab\quad\text{and}\quad\sqrt{\epsa}\diam(D)<\sqrt{\epsb}\lambda,\]
where $\diam(D)$ and $\lambda$ denote the diameter of $D$ and wavelength, respectively. Then, based on \cite{SKL}, it is possible to apply the first-order Born approximation $\mE_{\tot}(\mx,\ma_{n})\approx\mE_{\inc}(\mx,\ma_{n})$ to \eqref{ScatteringParameter} and correspondingly, we can observe that
\begin{equation}\label{Approximation1}
\mathcal{S}_{\meas}(n,m)\approx-\frac{ik_0^2}{4\omega\mub}\int_{D}\left(\frac{\epsa-\epsb}{\epsb}+i\frac{\sigmaa-\sigmab}{\omega\epsb}\right)\mE_{\inc}(\ma_{m},\mx)\cdot\mE_{\inc}(\mx,\ma_{n})\rd\mx.
\end{equation}
Let us emphasize that based on the configuration of the microwave machine, only the $z-$component of the incident and total fields can be handled because the antennas are arranged perpendicular to the $z-$axis, refer to \cite{SKLKLJC,SLP}. Then, by letting $z-$component of incident field $\mE_{\inc}$ as $\E$ and the reciprocity property of the $\mE_{\inc}$, \eqref{Approximation1} can be written by
\begin{equation}\label{Approximation}
\mathcal{S}_{\meas}(n,m)\approx-\frac{ik_0^2}{4\omega\mub}\int_{D}\mathcal{O}_D\E(\ma_{m},\mx)\E(\ma_{n},\mx)\rd\mx,
\end{equation}
where
\begin{equation}\label{ObjectIncident}
\mathcal{O}_D=\frac{\epsa-\epsb}{\epsb}+i\frac{\sigmaa-\sigmab}{\omega\epsb}\quad\text{and}\quad\E(\ma_n,\mx)=-\frac{i}{4}H_0^{(1)}(\kb|\ma_n-\mx|).
\end{equation}
Here, $H_0^{(1)}$ denotes the Hankel function of order zero of the first kind.

Next, we briefly introduce the imaging function of SM. Note that to introduce the imaging function, every element $\mathcal{S}_{\meas}(n,m)$ of the scattering matrix can be measured, i.e., the following scattering matrix must be available
\[\mathbb{K}=\begin{pmatrix}
\smallskip \mathcal{S}_{\meas}(1,1) & \mathcal{S}_{\meas}(1,2) & \cdots & \mathcal{S}_{\meas}(1,N-1) & \mathcal{S}_{\meas}(1,N) \\
\mathcal{S}_{\meas}(2,1) & \mathcal{S}_{\meas}(2,2) & \cdots & \mathcal{S}_{\meas}(1,N-1) & \mathcal{S}_{\meas}(2,N) \\
\smallskip \vdots & \vdots & \ddots & \vdots & \vdots \\
\mathcal{S}_{\meas}(N,1) & \mathcal{S}_{\meas}(N,2) & \cdots & \mathcal{S}_{\meas}(N,N-1) & \mathcal{S}_{\meas}(N,N)\end{pmatrix}.\]
However, under a certain simulation configuration, it is impossible to measure $\mathcal{S}_{\meas}(n,m)$ when $n=m$ because when an antenna $\Lambda_{m}$ is used for signal transmission, the remaining $N-1$ antennas $\Lambda_n$, $n=1,2,\cdots,N$ with $n\ne m$ are used for signal reception. For more discussion on this, refer to \cite{SKLKLJC}. Thus, in practice, we can use the following scattering matrix
\[\mathbb{K}=\begin{pmatrix}
\smallskip \text{unknown} & \mathcal{S}_{\meas}(1,2) & \cdots & \mathcal{S}_{\meas}(1,N-1) & \mathcal{S}_{\meas}(1,N) \\
\mathcal{S}_{\meas}(2,1) & \text{unknown} & \cdots & \mathcal{S}_{\meas}(1,N-1) & \mathcal{S}_{\meas}(2,N) \\
\smallskip \vdots & \vdots & \ddots & \vdots & \vdots \\
\mathcal{S}_{\meas}(N,1) & \mathcal{S}_{\meas}(N,2) & \cdots & \mathcal{S}_{\meas}(N,N-1) & \text{unknown}\end{pmatrix}.\]
and correspondingly, traditional SM cannot be applied directly. Instead of using $\mathbb{K}$, we consider the application of the following scattering matrix. For a constant $C\in\mathbb{C}$, let
\[\mathbb{K}(C)=\begin{pmatrix}
\smallskip C & \mathcal{S}_{\meas}(1,2) & \cdots & \mathcal{S}_{\meas}(1,N-1) & \mathcal{S}_{\meas}(1,N) \\
\mathcal{S}_{\meas}(2,1) & C & \cdots & \mathcal{S}_{\meas}(1,N-1) & \mathcal{S}_{\meas}(2,N) \\
\smallskip \vdots & \vdots & \ddots & \vdots & \vdots \\
\mathcal{S}_{\meas}(N,1) & \mathcal{S}_{\meas}(N,2) & \cdots & \mathcal{S}_{\meas}(N,N-1) & C\end{pmatrix}.\]
Then, for each search point $\ms\in\Omega$, based on the \eqref{Approximation} and $\mathbb{K}(C)$, define a unit test vector
\begin{equation}\label{TestVector}
\mW(\ms)=\frac{1}{\displaystyle\left(\sum_{n=1}^{N}|\E(\ma_n,\ms)|^2\right)^{1/2}}\begin{pmatrix}
\medskip\E(\ma_1,\ms)\\
\E(\ma_2,\ms)\\
\vdots\\
\E(\ma_N,\ms)
\end{pmatrix}
\end{equation}
and by performing SVD of $\mathbb{K}(C)$,
\[\mathbb{K}(C)=\sum_{n=1}^{N}\tau_n\mU_n\mV_n^*\approx\tau_1\mU_1\mV_1^*,\]
the imaging function can be introduced as
\begin{equation}\label{ImagingFunction}
\mathfrak{F}(\ms,C)=|\langle\mW(\ms),\mU_1\rangle\langle\mW(\ms),\overline{\mV}_1\rangle|,\quad\text{where}\quad\langle\mW(\ms),\mU_1\rangle=\mW(\ms)^*\mU_1.
\end{equation}
Then, for some specific selection of $C$, it is possible to recognize $\mx\in\Sigma$ through the map of $\mathfrak{F}(\ms,C)$.

\section{Analysis of the Imaging Function}\label{sec:3}
In this section, we explore the mathematical structure of the imaging function for a proper selection of the constant $C$ to guarantee a good result. The main result is as follows.

\begin{theorem}\label{TheoremStructure}
Let $\vt_n=\ma_n/R=(\cos\theta_n,\sin\theta_n)$, $\ms=|\ms|(\cos\varphi,\sin\varphi)$, and $\ms-\mx=|\ms-\mx|(\cos\phi,\sin\phi)$. If $|k(\ms-\ma_n)|\gg0.25$ for $\ms\in\Omega$, then $\mathfrak{F}(\ms,C)$ can be represented as follows:
\begin{equation}\label{StructureImagingFunction}
\mathfrak{F}(\ms,C)=|\Phi_1(\ms)+\Phi_2(\ms)+\Phi_3(\ms,C)|,
\end{equation}
where
\begin{align*}
\Phi_1(\ms)&=\frac{N\omega\epsb e^{2iR\kb}\mathcal{O}_D}{32R\kb\tau_1\pi}\int_D\bigg(J_0(\kb|\ms-\mx|)+\frac{1}{N}\sum_{n=1}^{N}\sum_{s\in\mathbb{Z}_0^*}i^sJ_s(\kb|\ms-\mx|)e^{is(\theta_n-\phi)}\bigg)^2\rd\mx\\
\Phi_2(\ms)&=-\frac{\omega\epsb e^{2iR\kb}\mathcal{O}_D}{32R\kb\tau_1\pi}\int_D\bigg(J_0(2\kb|\ms-\mx|)+\frac{1}{N}\sum_{n=1}^{N}\sum_{s\in\mathbb{Z}_0^*}i^sJ_s(2\kb|\ms-\mx|)e^{is(\theta_n-\phi)}\bigg)\rd\mx\\
\Phi_3(\ms,C)&=\frac{C}{\tau_1}\bigg(J_0(2\kb|\ms|)+\frac{1}{N}\sum_{n=1}^{N}\sum_{s\in\mathbb{Z}_0^*}i^sJ_s(2\kb|\ms|)e^{is\varphi}\bigg).
\end{align*}
Here, $\mathbb{Z}_0^*=\mathbb{Z}\cup\set{-\infty,+\infty}\backslash\set{0}$ and $\mathbb{Z}$ denotes the set of integer number and $J_s$ is the Bessel function of order $s$ of the first kind.
\end{theorem}
\begin{proof}
Since $\ms\in\Omega$ and $|k(\ms-\ma_n)|\gg0.25$ for all $n$, the Hankel function can be given as the following asymptotic form (see \cite[Theorem 2.5]{CK} for instance)
\[H_0^{(1)}(\kb|\ms-\ma_n|)\approx\frac{(1+i)e^{i\kb|\ma_n|}}{\sqrt{\kb\pi|\ma_n|}}e^{-i\kb\vt_n\cdot\ms}.\]
By using this representation, the test vector of \eqref{TestVector} and the measurement data \eqref{Approximation} can be written as
\[\mW(\ms)=\frac{1}{\sqrt{N}}\bigg(e^{-i\kb\vt_1\cdot\ms},e^{-i\kb\vt_2\cdot\ms},\ldots,e^{-i\kb\vt_N\cdot\ms}\bigg)^T\]
and 
\[\mathcal{S}_{\meas}(n,m)\approx-\frac{k_0^2e^{2i|\ma_n|\kb}\mathcal{O}_D}{32|\ma_n|\kb\omega\mub\pi}\int_De^{-i\kb(\vt_m+\vt_n)\cdot\mx}\rd\mx:=A\int_D e^{-i\kb(\vt_m+\vt_n)\cdot\mx}\rd\mx,\]
respectively. Since
\[\mU_1\mV_1^*=\frac{1}{\tau_1}\mathbb{K}(C)=\frac{1}{\tau_1}\begin{pmatrix}
C & \mathcal{S}_{\meas}(1,2) & \cdots & \mathcal{S}_{\meas}(1,N-1) & \mathcal{S}_{\meas}(1,N) \\
\mathcal{S}_{\meas}(2,1) & C & \cdots & \mathcal{S}_{\meas}(1,N-1) & \mathcal{S}_{\meas}(2,N) \\
\vdots & \vdots & \ddots & \vdots & \vdots \\
\mathcal{S}_{\meas}(N,1) & \mathcal{S}_{\meas}(N,2) & \cdots & \mathcal{S}_{\meas}(N,N-1) & C\end{pmatrix},\]
we can examine that
\begin{align*}
\mW(\ms)^*\mU_1\mV_1^*&=\frac{1}{\tau_1\sqrt{N}}\begin{pmatrix}
\smallskip e^{i\kb\vt_1\cdot\ms}\\
e^{i\kb\vt_2\cdot\ms}\\
\medskip\vdots\\
e^{i\kb\vt_N\cdot\ms}
\end{pmatrix}^T\begin{pmatrix}
\smallskip C & \mathcal{S}_{\meas}(1,2) & \cdots & \mathcal{S}_{\meas}(1,N-1) & \mathcal{S}_{\meas}(1,N) \\
\mathcal{S}_{\meas}(2,1) & C & \cdots & \mathcal{S}_{\meas}(1,N-1) & \mathcal{S}_{\meas}(2,N) \\
\smallskip\vdots & \vdots & \ddots & \vdots & \vdots \\
\mathcal{S}_{\meas}(N,1) & \mathcal{S}_{\meas}(N,2) & \cdots & \mathcal{S}_{\meas}(N,N-1) & C\end{pmatrix}\\
&=\frac{1}{\tau_1\sqrt{N}}
\begin{pmatrix}
\displaystyle\medskip Ce^{i\kb\vt_1\cdot\ms}+A\int_De^{-i\kb\vt_1\cdot\mx}\sum_{n\in\mathcal{N}_1}e^{i\kb\vt_n\cdot(\ms-\mx)}\rd\mx\\
\displaystyle Ce^{i\kb\vt_2\cdot\ms}+A\int_De^{-i\kb\vt_2\cdot\mx}\sum_{n\in\mathcal{N}_2}e^{i\kb\vt_n\cdot(\ms-\mx)}\rd\mx\\
\medskip\vdots\\
\displaystyle Ce^{i\kb\vt_N\cdot\ms}+A\int_De^{-i\kb\vt_N\cdot\mx}\sum_{n\in\mathcal{N}_N}e^{i\kb\vt_N\cdot(\ms-\mx)}\rd\mx
\end{pmatrix}^T,
\end{align*}
where $\mathcal{N}_n=\set{1,2,\cdots,N}\backslash\set{n}$.

Since the following relation holds uniformly (see \cite{P-SUB3} for derivation)
\begin{align}
\begin{aligned}\label{JacobiAnger}
\sum_{n=1}^{N}e^{i\kb\vt_n\cdot(\ms-\mx)}&=\sum_{n=1}^{N}e^{i\kb|\ms-\mx|\cos(\theta_n-\phi)}=\sum_{n=1}^{N}\bigg(J_0(\kb|\ms-\mx|)+\sum_{s\in\mathbb{Z}_0^*}i^sJ_s(\kb|\ms-\mx|)e^{is(\theta_n-\phi)}\bigg)\\
&=NJ_0(\kb|\ms-\mx|)+\sum_{n=1}^{N}\sum_{s\in\mathbb{Z}_0^*}i^sJ_s(\kb|\ms-\mx|)e^{is(\theta_n-\phi)}:=\mathcal{J}(\mx,\ms),
\end{aligned}
\end{align}
we can obtain
\begin{align*}
\mW(\ms)^*\mU_1\mV_1^*&=\frac{1}{\tau_1\sqrt{N}}
\begin{pmatrix}
\displaystyle\medskip Ce^{i\kb\vt_1\cdot\ms}+A\int_De^{-i\kb\vt_1\cdot\mx}\sum_{n=1}^{N}e^{i\kb\vt_n\cdot(\ms-\mx)}\rd\mx-A\int_De^{-i\kb\vt_1\cdot\mx}e^{i\kb\vt_1\cdot(\ms-\mx)}\rd\mx\\
\displaystyle Ce^{i\kb\vt_2\cdot\ms}+A\int_De^{-i\kb\vt_2\cdot\mx}\sum_{n=1}^{N}e^{i\kb\vt_n\cdot(\ms-\mx)}\rd\mx-A\int_De^{-i\kb\vt_2\cdot\mx}e^{i\kb\vt_2\cdot(\ms-\mx)}\rd\mx\\\
\medskip\vdots\\
\displaystyle Ce^{i\kb\vt_N\cdot\ms}+A\int_De^{-i\kb\vt_N\cdot\mx}\sum_{n=1}^{N}e^{i\kb\vt_n\cdot(\ms-\mx)}\rd\mx-A\int_De^{-i\kb\vt_N\cdot\mx}e^{i\kb\vt_N\cdot(\ms-\mx)}\rd\mx\
\end{pmatrix}^T\\
&=\frac{1}{\tau_1\sqrt{N}}\begin{pmatrix}
\displaystyle\medskip Ce^{i\kb\vt_1\cdot\ms}+A\int_De^{-i\kb\vt_1\cdot\mx}\mathcal{J}(\mx,\ms)\rd\mx-A\int_De^{-i\kb\vt_1\cdot\mx}e^{i\kb\vt_1\cdot(\ms-\mx)}\rd\mx\\
\displaystyle Ce^{i\kb\vt_2\cdot\ms}+A\int_De^{-i\kb\vt_2\cdot\mx}\mathcal{J}(\mx,\ms)\rd\mx-A\int_De^{-i\kb\vt_2\cdot\mx}e^{i\kb\vt_2\cdot(\ms-\mx)}\rd\mx\\\
\medskip\vdots\\
\displaystyle Ce^{i\kb\vt_N\cdot\ms}+A\int_De^{-i\kb\vt_N\cdot\mx}\mathcal{J}(\mx,\ms)\rd\mx-A\int_De^{-i\kb\vt_N\cdot\mx}e^{i\kb\vt_N\cdot(\ms-\mx)}\rd\mx\
\end{pmatrix}^T
\end{align*}
and correspondingly,
\begin{align*}
&\mW(\ms)^*\mU_1\mV_1^*\overline{\mW}(\ms)\\
&=\frac{1}{\tau_1 N}
\begin{pmatrix}
\displaystyle\medskip Ce^{i\kb\vt_1\cdot\ms}+A\int_De^{-i\kb\vt_1\cdot\mx}\mathcal{J}(\mx,\ms)\rd\mx-A\int_De^{-i\kb\vt_1\cdot\mx}e^{i\kb\vt_1\cdot(\ms-\mx)}\rd\mx\\
\displaystyle Ce^{i\kb\vt_2\cdot\ms}+A\int_De^{-i\kb\vt_2\cdot\mx}\mathcal{J}(\mx,\ms)\rd\mx-A\int_De^{-i\kb\vt_2\cdot\mx}e^{i\kb\vt_2\cdot(\ms-\mx)}\rd\mx\\\
\medskip\vdots\\
\displaystyle Ce^{i\kb\vt_N\cdot\ms}+A\int_De^{-i\kb\vt_1\cdot\mx}\mathcal{J}(\mx,\ms)\rd\mx-A\int_De^{-i\kb\vt_N\cdot\mx}e^{i\kb\vt_N\cdot(\ms-\mx)}\rd\mx\
\end{pmatrix}^T
\begin{pmatrix}
\smallskip e^{i\kb\vt_1\cdot\ms}\\
e^{i\kb\vt_2\cdot\ms}\\
\medskip\vdots\\
e^{i\kb\vt_N\cdot\ms}
\end{pmatrix}\\
&=\frac{1}{\tau_1 N}\sum_{n=1}^{N}\left(Ce^{2i\kb\vt_n\cdot\ms}+A\int_De^{i\kb\vt_n\cdot(\ms-\mx)}\mathcal{J}(\mx,\ms)\rd\mx-A\int_De^{2i\kb\vt_n\cdot(\ms-\mx)}\rd\mx\right).
\end{align*}
Now, by applying \eqref{JacobiAnger} again, we obtain
\begin{align*}
&\langle\mW(\ms),\mU_1\rangle\langle\mW(\ms),\overline{\mV}_1\rangle=\mW(\ms)^*\mU_1\mV_1^*\overline{\mW}(\ms)\\
&=\frac{C}{\tau_1 N}\sum_{n=1}^{N}e^{2i\kb\vt_n\cdot\ms}+\frac{A}{\tau_1 N}\int_D\mathcal{J}(\mx,\ms)\sum_{n=1}^{N}e^{i\kb\vt_n\cdot(\ms-\mx)}\rd\mx-\frac{A}{\tau_1 N}\int_D\sum_{n=1}^{N}e^{2i\kb\vt_N\cdot(\ms-\mx)}\rd\mx\\
&=\frac{A}{\tau_1 N}\int_D\bigg(NJ_0(\kb|\ms-\mx|)+\sum_{n=1}^{N}\sum_{s\in\mathbb{Z}_0^*}i^sJ_s(\kb|\ms-\mx|)e^{is(\theta_n-\phi)}\bigg)^2\rd\mx\\
&-\frac{A}{\tau_1 N}\int_D\bigg(NJ_0(2\kb|\ms-\mx|)+\sum_{n=1}^{N}\sum_{s\in\mathbb{Z}_0^*}i^sJ_s(2\kb|\ms-\mx|)e^{is(\theta_n-\phi)}\bigg)\rd\mx\\
&+\frac{C}{\tau_1 N}\bigg(NJ_0(2\kb|\ms|)+\sum_{n=1}^{N}\sum_{s\in\mathbb{Z}_0^*}i^sJ_s(2\kb|\ms|)e^{is\varphi}\bigg).
\end{align*}
With this, we obtain the structure \eqref{StructureImagingFunction}.
\end{proof}

Based on Theorem \ref{TheoremStructure}, we can explain some properties of the imaging function and phenomena that can occur in simulation results.

\begin{remark}[Detectability of the object]\label{Remark1}
Since $\Phi_1(\ms)$ and $\Phi(\ms)$ contain the factors $J_0(\kb|\ms-\mx|)$ and $J_0(2\kb|\ms-\mx|)$, respectively, the map of $\mathfrak{F}(\ms,C)$ will contain peak of large magnitude when $\ms=\mx\in D$ so that it will be possible to recognize the location and shape of $D$.
\end{remark}

\begin{remark}[Impact of the constant $C$ and its best selection]\label{Remark2}
Since $\Phi_3(\ms,C)$ contains no information of $\mx\in D$, this factor contributes to generating several artifacts, and correspondingly, it will disturb the detection of $D$. Therefore, the application of nonzero constant $C$ will make it difficult to retrieve objects via the map of $\mathfrak{F}(\ms,C)$, so the selection $C=0$ will guarantee good imaging results. This is the theoretical reason why the diagonal elements of the scattering matrix are set to zero in most studies. In this case, since $J_0(0)=1$ and $\mathfrak{F}(\ms,C)\approx1$ when $\ms\in D$,
\[\mathfrak{F}(\ms,0)\approx\abs{\frac{N\omega\epsb e^{2iR\kb}\mathcal{O}_D}{32R\kb\tau_1\pi}\area(D)-\frac{\omega\epsb e^{2iR\kb}\mathcal{O}_D}{32R\kb\tau_1\pi}\area(D)}\approx 1.\]
Correspondingly, $\mathfrak{F}(\ms,0)$ can be represented as
\begin{multline*}
\mathfrak{F}(\ms,0)\approx\frac{N}{(N-1)\area(D)}\int_D\bigg(J_0(\kb|\ms-\mx|)+\frac{1}{N}\sum_{n=1}^{N}\sum_{s\in\mathbb{Z}_0^*}i^sJ_s(\kb|\ms-\mx|)e^{is(\theta_n-\phi)}\bigg)^2\rd\mx\\
-\frac{1}{(N-1)\area(D)}\int_D\bigg(J_0(2\kb|\ms-\mx|)+\frac{1}{N}\sum_{n=1}^{N}\sum_{s\in\mathbb{Z}_0^*}i^sJ_s(2\kb|\ms-\mx|)e^{is(\theta_n-\phi)}\bigg)\rd\mx.
\end{multline*}
This is the same result derived in \cite{P-SUB11}. Hence, we can say that Theorem \ref{TheoremStructure} can be regarded as a generalized version of the previous study.
\end{remark}

\begin{remark}[Ideal conditions to obtain good results]\label{Remark3}
Based on Remarks \ref{Remark1} and \ref{Remark2}, we can obtain good results when $\Phi_3(\ms,C)=0$, i.e., $C=0$,
\begin{equation}\label{EliminateTerms}
\frac{1}{N}\sum_{n=1}^{N}\sum_{s\in\mathbb{Z}_0^*}i^sJ_s(\kb|\ms-\mx|)e^{is(\theta_n-\phi)}=0,\quad\text{and}\quad\frac{1}{N}\sum_{n=1}^{N}\sum_{s\in\mathbb{Z}_0^*}i^sJ_s(2\kb|\ms-\mx|)e^{is(\theta_n-\phi)}=0.
\end{equation}
The easiest and simplest method is to select $N\longrightarrow+\infty$; however, this approach is impossible considering the simulation circumstance. Since the following asymptotic form holds for large $z$
\[J_s(z)=\sqrt{\frac{2}{\pi z}}\cos\left(z-\frac{s\pi}{2}-\frac{\pi}{4}\right)\quad\text{implies}\quad J_s(\kb|\ms-\mx|)\propto\sqrt{\frac{2}{\kb\pi|\ms-\mx|}},\]
it will be possible to make $J_s(\kb|\ms-\mx|),J_s(2\kb|\ms-\mx|)\longrightarrow+\infty$ when $|\kb|\longrightarrow+\infty$ or equivalently $f\longrightarrow+\infty$. However, applying infinite-valued frequency is impossible in practice. Based on these observations, we can confirm that if the total number $N$ is small or the applied frequency $f$ is low, it will not be easy to identify the object.
\end{remark}

\begin{remark}[Practical condition to obtain good results]\label{Remark4} 
Here, we consider finding a condition that satisfies the following relation: for $x\ne0$ and sufficiently large $L$,
\[\mathcal{E}(x,L)=\sum_{n=1}^{N}\sum_{s=-L,s\ne0}^{L}i^sJ_s(x)e^{is(\theta_n-\phi)}\approx0.\]
Note that if $N=16$ and all antennas are uniformly distributed on a circle of radius $R$ such as the simulation configuration in Figure \ref{Simulation_Configration} of Section \ref{sec:4}, then $\mathcal{E}(x,L)\approx0$ for small $L$, refer to \cite{P-SUB16}. If $L$ is sufficiently large, then since (see \cite{L} for instance)
\[|J_L(x)|\leq\frac{b}{\sqrt[3]{L}},\quad b=0.674885\ldots\quad\text{implies}\quad\mathcal{E}(x,L)\approx0.\]
Table \ref{table} presents the values of $|\mathcal{E}(x,L)|$ with various $x$ and $L$. This is the theoretical reason why uniformly distributed antenna arrangement was used in most studies.
\end{remark}

\begin{table}[h]
\begin{center}
\begin{tabular}{c||c|c|c|c|c}
\hline \centering ~&$L=1$&$L=3$&$L=5$&$L=10$&$L\geq15$\\
\hline\centering $x=0.1$&$\SI{1.5266e-16}{}$&$\SI{1.3877e-16}{}$&$\SI{1.8049e-16}{}$&$\SI{2.2053e-16}{}$&$\SI{2.2053e-16}{}$\\
\hline\centering $x=0.3$&$\SI{3.3307e-16}{}$&$\SI{3.1772e-16}{}$&$\SI{4.8908e-16}{}$&$\SI{3.1406e-16}{}$&$\SI{3.0425e-16}{}$\\
\hline\centering $x=0.5$&$\SI{5.5511e-17}{}$&$\SI{4.6230e-16}{}$&$\SI{7.4906e-16}{}$&$\SI{7.9050e-16}{}$&$\SI{6.6733e-16}{}$\\
\hline\centering $x=0.7$&$\SI{3.3307e-16}{}$&$\SI{9.8275e-16}{}$&$\SI{5.5563e-16}{}$&$\SI{9.3646e-16}{}$&$\SI{7.0292e-16}{}$\\
\hline\centering $x=1.0$&$\SI{5.5511e-16}{}$&$\SI{2.1957e-16}{}$&$\SI{4.6286e-16}{}$&$\SI{1.5563e-16}{}$&$\SI{1.6795e-16}{}$\\
\hline
\end{tabular}
\caption{\label{table}Values of $|\mathcal{E}(x,L)|$ for various $x$ and $L$.}
\end{center}
\end{table}

Based on Remarks \ref{Remark1} and \ref{Remark4}, we can obtain the following unique identification result.

\begin{corollary}[Unique identification of object]
Assume that $N$ is an even number greater than $8$ and all antennas $\Lambda_n$, $n=1,2,\cdots,N$, are uniformly distributed on a circle of radius $R$. Then, for sufficiently high frequency $f$ and sufficiently small constant $|C|$, the object $D$ can be retrieved uniquely through the maps of $\mathfrak{F}(\ms,0)$ and $\mathfrak{F}(\ms,C)$.
\end{corollary}

\section{Simulation Results with Synthetic and Experimental Data and Discussions}\label{sec:4}
\subsection{Simulation Results with Synthetic Data}
In this section, we present simulation results to demonstrate the theoretical result. To perform the simulation, we set the location of dipole antennas for $n=1,2,\cdots,N(=16)$, 
\[\ma_n=\SI{0.09}{\meter}\left(\cos\theta_n,\sin\theta_n\right),\quad\theta_n=\frac{3\pi}{2}-\frac{2\pi(n-1)}{N}.\]
The ROI $\Omega$ was selected as the interior of a square region $(-\SI{0.08}{\meter},\SI{0.08}{\meter})^2$ with material properties $(\epsb,\sigmab)=(20\eps_0,\SI{0.2}{\siemens/\meter})$ at $f=\SI{1}{\giga\hertz}$. Here, $\eps_0=\SI{8.854e-12}{\farad/\meter}$ denotes the vacuum permittivity. We selected two balls $D_1$ and $D_2$ as objects with the same radii $\alpha=\SI{0.01}{\meter}$, centers $\mx_1=(\SI{0.01}{\meter},\SI{0.03}{\meter})$ and $\mx_2=(-\SI{0.04}{\meter},-\SI{0.02}{\meter})$, and material properties $(\eps_1,\sigma_1)=(55\eps_0,\SI{1.2}{\siemens/\meter})$ and $(\eps_2,\sigma_2)=(45\eps_0,\SI{1.0}{\siemens/\meter})$. Figure \ref{Simulation_Configration} illustrates the simulation settings without and with objects.

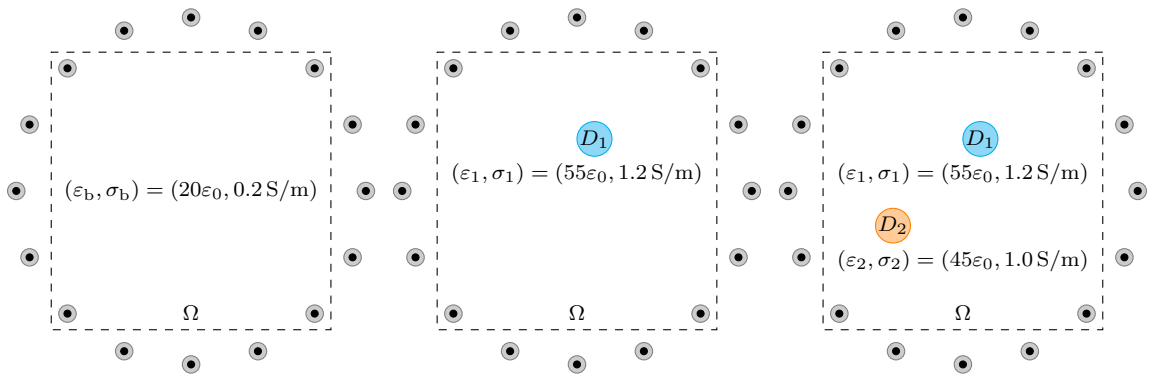
\begin{figure}[h]
\begin{center}
\begin{tikzpicture}[scale=2.3]
\draw[black,dashed] (0.8,0.8) -- (-0.8,0.8) -- (-0.8,-0.8) -- (0.8,-0.8) -- cycle;
\node at (0,-0.7) {\footnotesize$\Omega$};
\foreach \alpha in {0,22.5,...,337.5}
{\draw[gray,fill=gray!40!white] ({cos(\alpha)},{sin(\alpha)}) circle (0.05cm);
\draw[black,fill=black] ({cos(\alpha)},{sin(\alpha)}) circle (0.02cm);}
\node at (0,0) {\footnotesize$(\epsb,\sigmab)=(20\eps_0,\SI{0.2}{\siemens/\meter})$};
\end{tikzpicture}\hfill
\begin{tikzpicture}[scale=2.3]
\draw[black,dashed] (0.8,0.8) -- (-0.8,0.8) -- (-0.8,-0.8) -- (0.8,-0.8) -- cycle;
\node at (0,-0.7) {\footnotesize$\Omega$};
\foreach \alpha in {0,22.5,...,337.5}
{\draw[gray,fill=gray!40!white] ({cos(\alpha)},{sin(\alpha)}) circle (0.05cm);
\draw[black,fill=black] ({cos(\alpha)},{sin(\alpha)}) circle (0.02cm);}
\draw[cyan,fill=cyan!40!white] (0.1,0.3) circle (0.1cm);
\node at (0.1,0.3) {\footnotesize$D_1$};
\node at (0,0.1) {\footnotesize$(\eps_1,\sigma_1)=(55\eps_0,\SI{1.2}{\siemens/\meter})$};
\end{tikzpicture}\hfill
\begin{tikzpicture}[scale=2.3]
\draw[black,dashed] (0.8,0.8) -- (-0.8,0.8) -- (-0.8,-0.8) -- (0.8,-0.8) -- cycle;
\node at (0,-0.7) {\footnotesize$\Omega$};
\foreach \alpha in {0,22.5,...,337.5}
{\draw[gray,fill=gray!40!white] ({cos(\alpha)},{sin(\alpha)}) circle (0.05cm);
\draw[black,fill=black] ({cos(\alpha)},{sin(\alpha)}) circle (0.02cm);}
\draw[cyan,fill=cyan!40!white] (0.1,0.3) circle (0.1cm);
\node at (0.1,0.3) {\footnotesize$D_1$};
\node at (0,0.1) {\footnotesize$(\eps_1,\sigma_1)=(55\eps_0,\SI{1.2}{\siemens/\meter})$};
\draw[orange,fill=orange!40!white] (-0.4,-0.2) circle (0.1cm);
\node at (-0.4,-0.2) {\footnotesize$D_2$};
\node at (0,-0.4) {\footnotesize$(\eps_2,\sigma_2)=(45\eps_0,\SI{1.0}{\siemens/\meter})$};
\end{tikzpicture}
\caption{\label{Simulation_Configration} Illustration of the background (left), single (center) and multiple (right) small objects.}
\end{center}
\end{figure}

\begin{example}[Imaging of single object]\label{ex1}
In this example, we consider the imaging of a single object $D_1$. Figure \ref{SV1} shows the distribution of singular values of $\mathbb{K}(C)$ with various $C$. We observe that it is very easy to discriminate nonzero singular value $\tau_1$ when $C=0$, $0.01$, $0.001i$, and $0.01i$. However, if the value of $|C|$ is not close to zero, it is not easy to discriminate nonzero and close to zero singular values, e.g., $C=0.1$ and $C=0.1i$. Fortunately, by regarding differences $\tau_n-\tau_{n+1}$, $n=1,2,\cdots,15$, we selected $\tau_1$ as the nonzero singular value of $\mathbb{K}(C)$ and defined the imaging function $\mathfrak{F}(\ms,C)$ of \eqref{ImagingFunction}.

Figure \ref{Result1} shows the maps of $\mathfrak{F}(\ms,C)$ with various $C$. Note that if $C=0$ or $|C|$ is small ($|C|\leq0.01$), the location and shape of $D_1$ can be identified accurately through the map of $\mathfrak{F}(\ms,C)$. Meanwhile, it is not easy to distinguish $D_1$ and an artifact with large magnitude at $(-\SI{0.01}{\meter},-\SI{0.03}{\meter})$ if $C=0.1$. Although some artifacts are included in the imaging result, it is possible to identify $D_1$ because the magnitudes of artifacts are small when $C=0.1i$. Thus, we can conclude that the imaging result of $\mathfrak{F}(\ms,C)$ significantly depends on the choice of $C$. Additionally, good results can be obtained when $C=0$ or $|C|$ is sufficiently small ($|C|\leq0.01$ in this example) for identifying single object.
\end{example}

\begin{figure}[h]
\subfigure[$C=0$]{\includegraphics[width=0.33\textwidth]{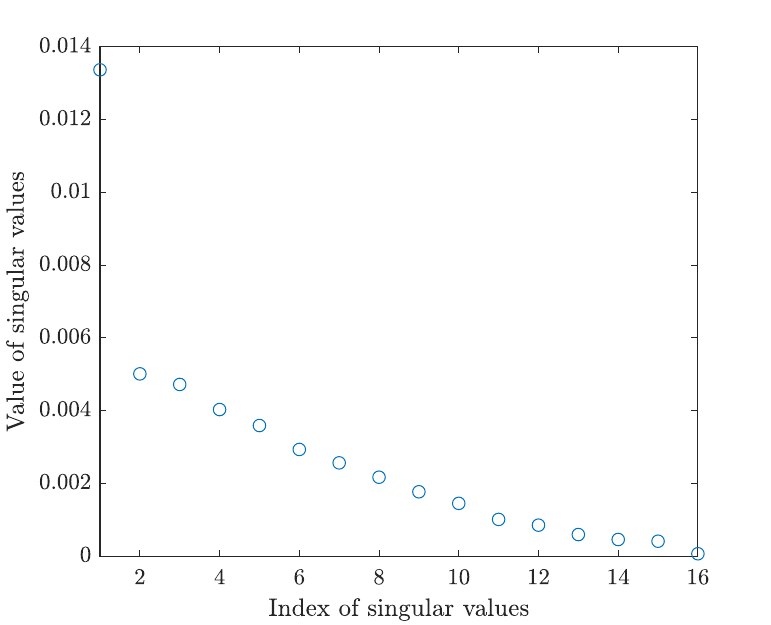}}\hfill
\subfigure[$C=0.01$]{\includegraphics[width=0.33\textwidth]{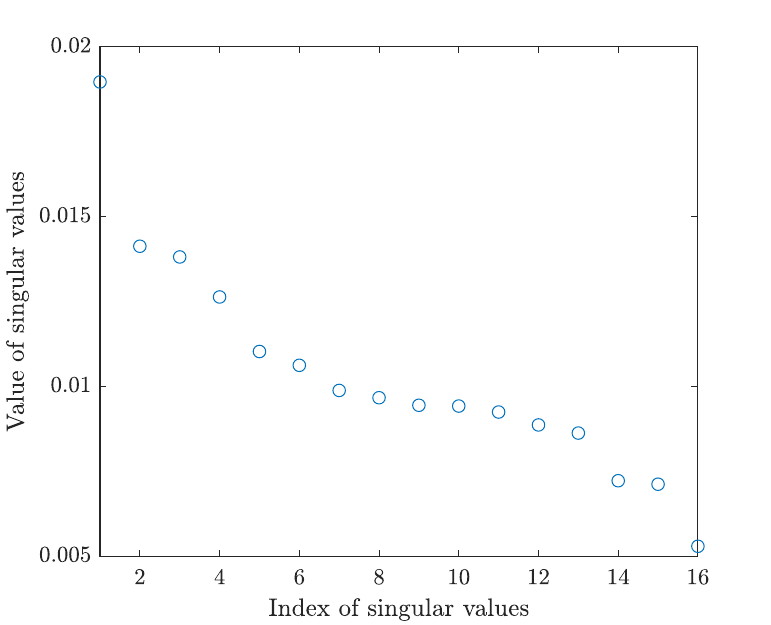}}\hfill
\subfigure[$C=0.1$]{\includegraphics[width=0.33\textwidth]{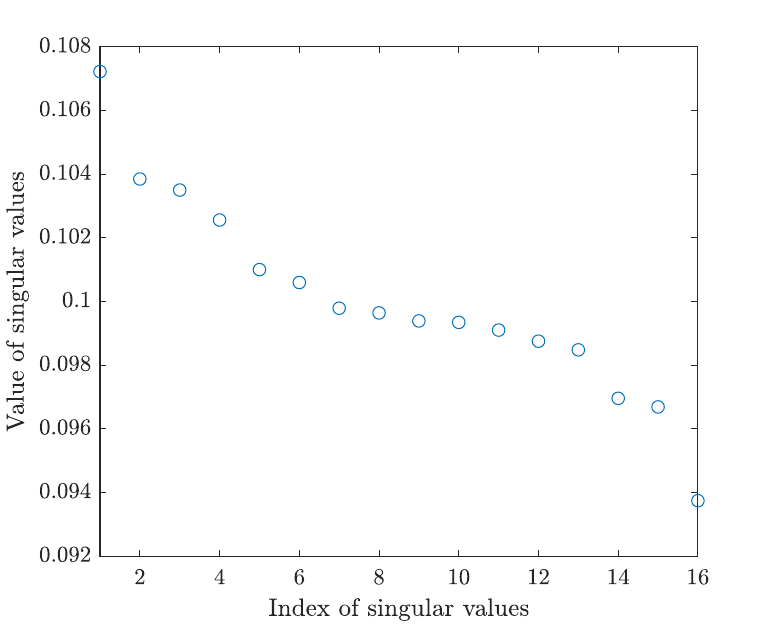}}\\
\subfigure[$C=0.001i$]{\includegraphics[width=0.33\textwidth]{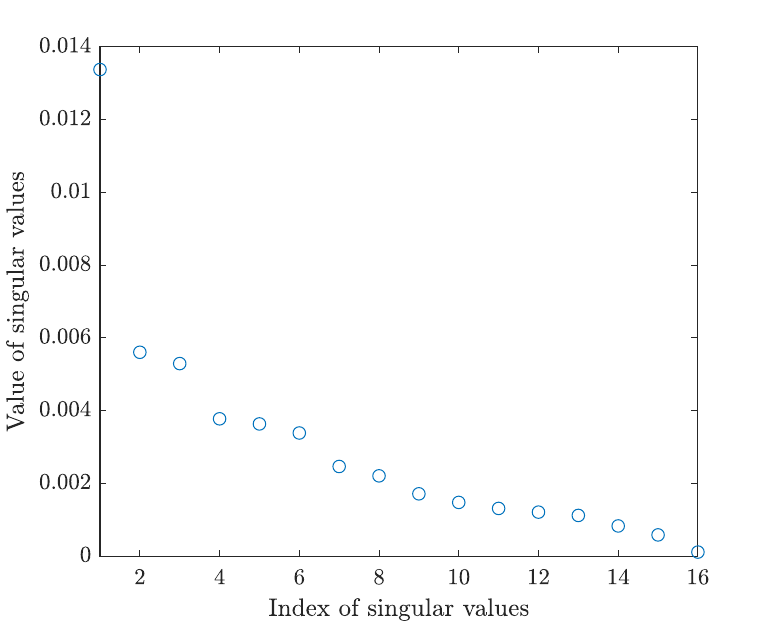}}\hfill
\subfigure[$C=0.01i$]{\includegraphics[width=0.33\textwidth]{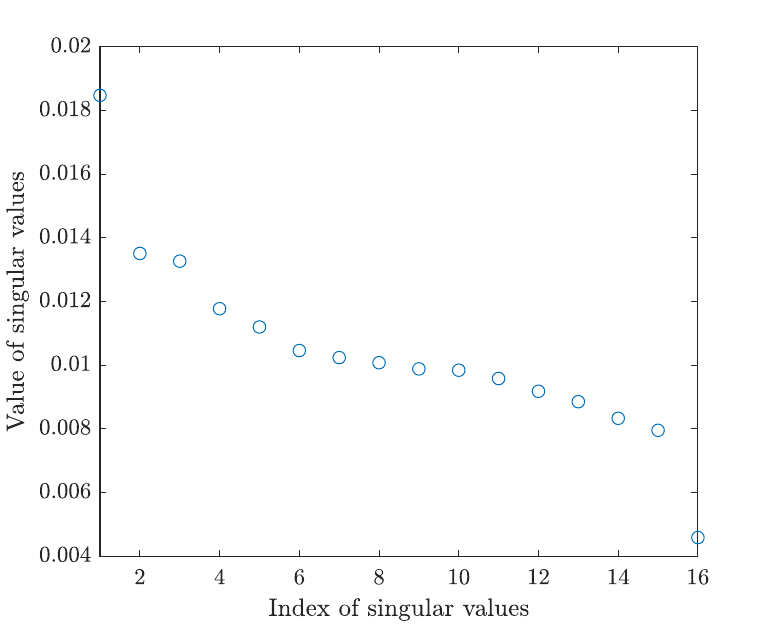}}\hfill
\subfigure[$C=0.1i$]{\includegraphics[width=0.33\textwidth]{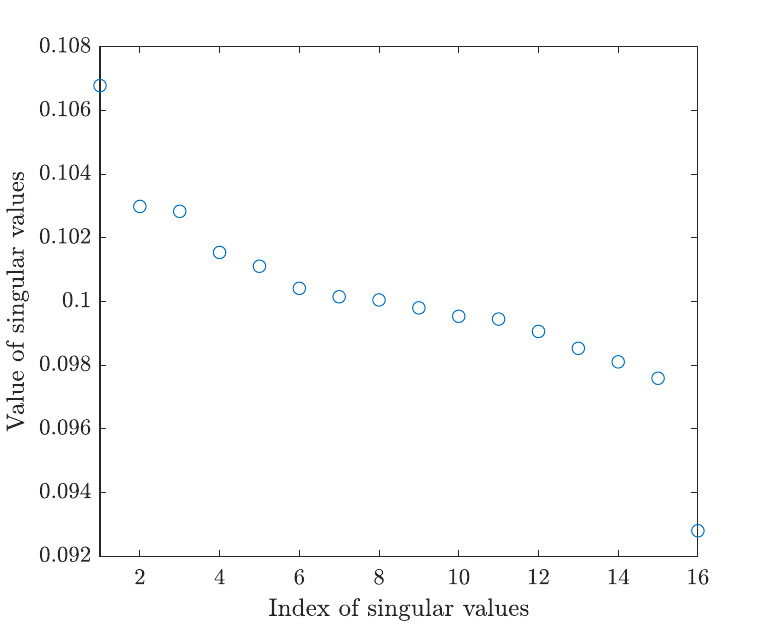}}
\caption{\label{SV1}(Example \ref{ex1}) Distribution of the singular values of $\mathbb{K}(C)$ at $f=\SI{1}{\giga\hertz}$ with various $C$.}
\end{figure}

\begin{figure}[h]
\subfigure[$C=0$]{\includegraphics[width=0.33\textwidth]{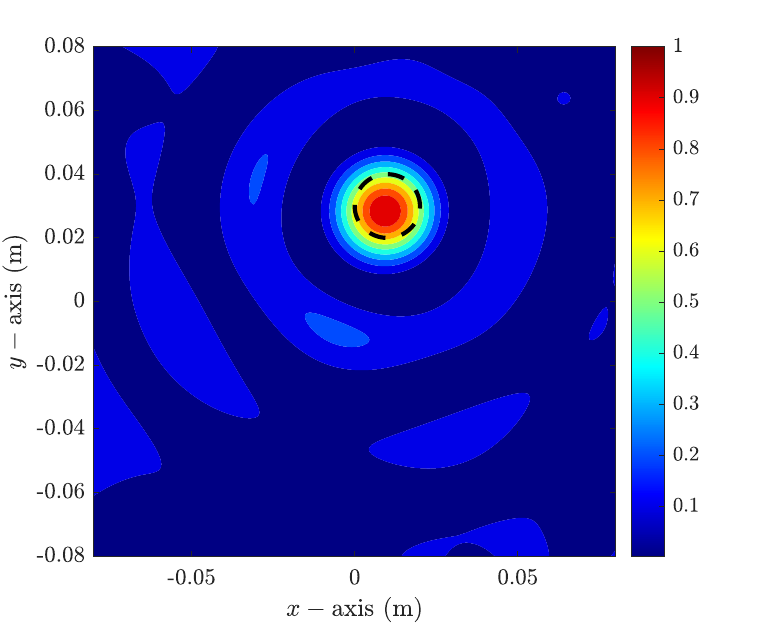}}\hfill
\subfigure[$C=0.01$]{\includegraphics[width=0.33\textwidth]{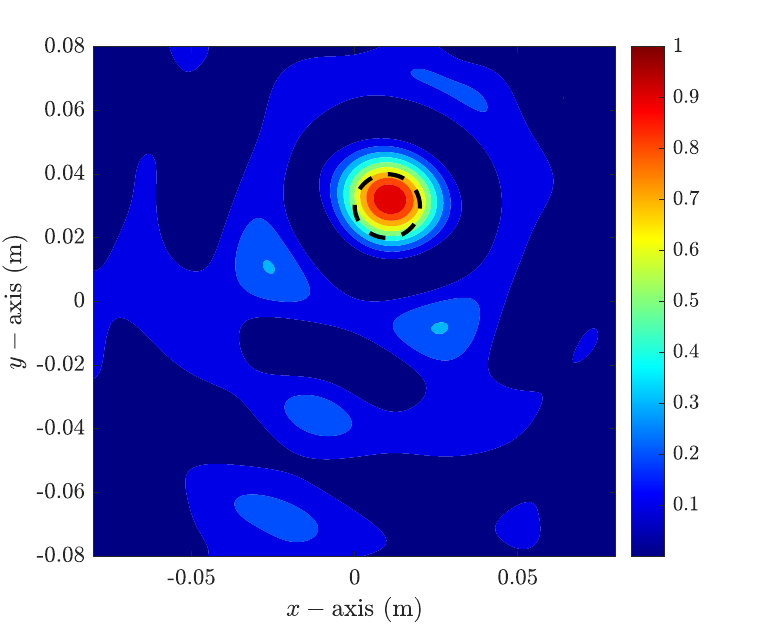}}\hfill
\subfigure[$C=0.1$]{\includegraphics[width=0.33\textwidth]{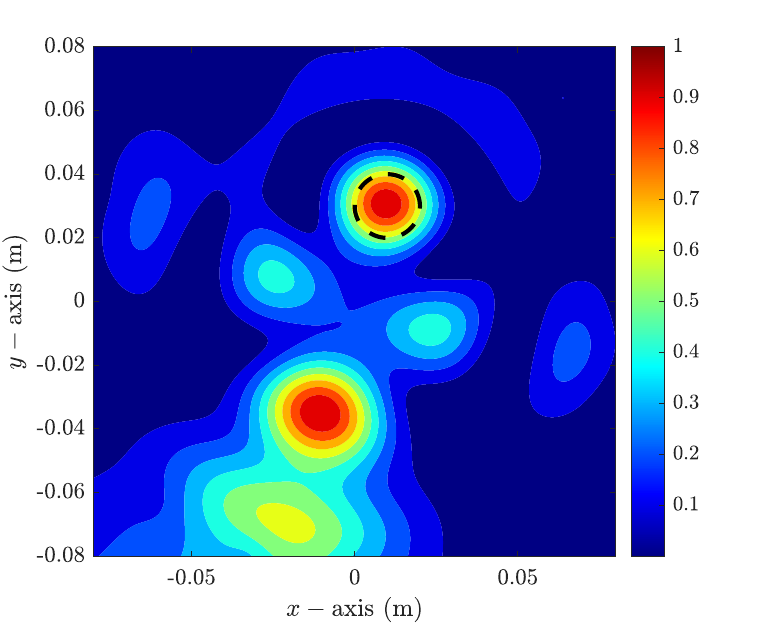}}\hfill\\
\subfigure[$C=0.001i$]{\includegraphics[width=0.33\textwidth]{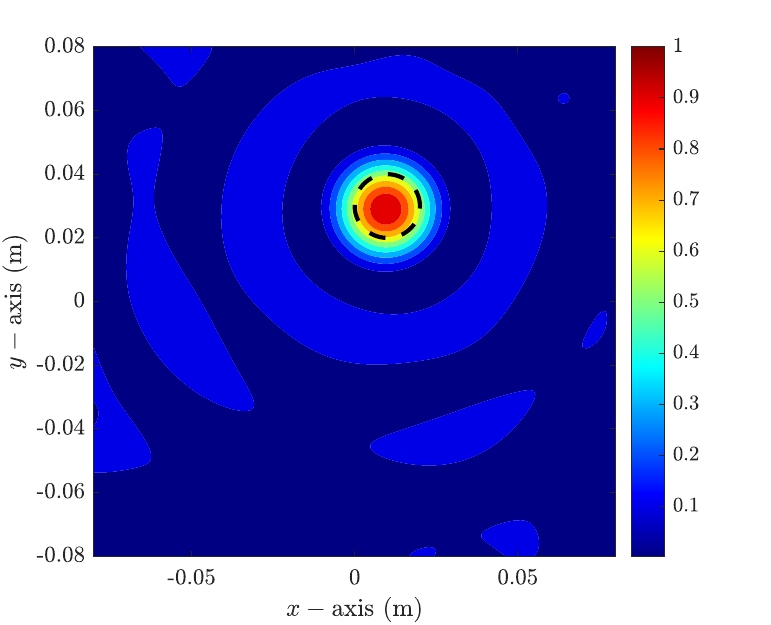}}\hfill
\subfigure[$C=0.01i$]{\includegraphics[width=0.33\textwidth]{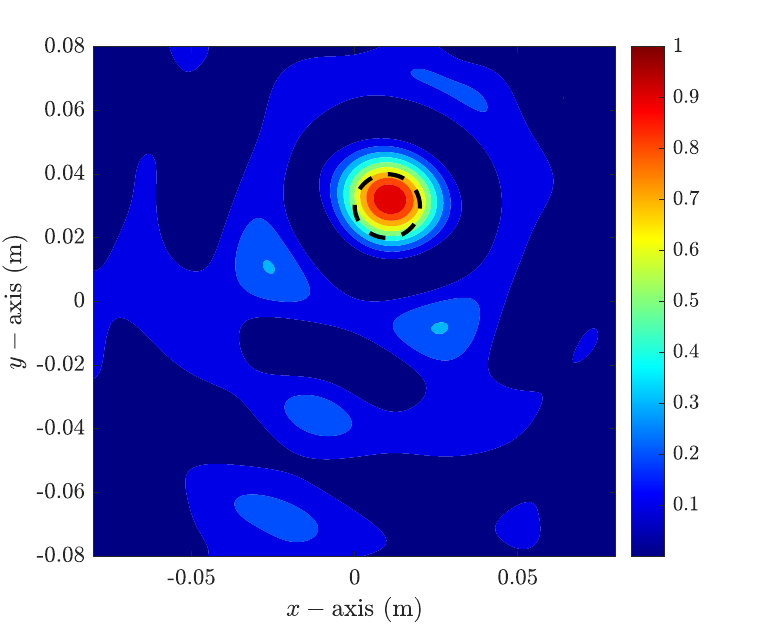}}\hfill
\subfigure[$C=0.1i$]{\includegraphics[width=0.33\textwidth]{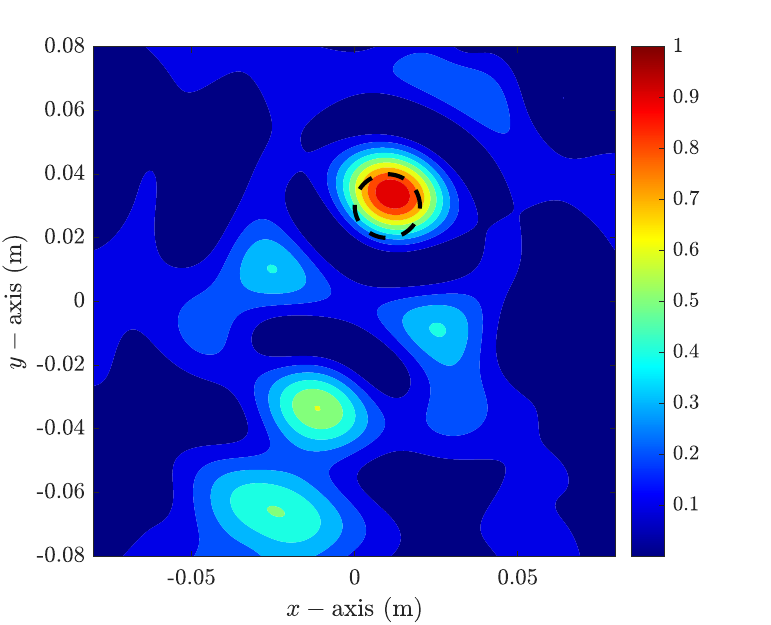}}\hfill
\caption{\label{Result1}(Example \ref{ex1}) Maps of $\mathfrak{F}(\ms,C)$ at $f=\SI{1}{\giga\hertz}$ with various $C$. Black-colored dashed line describes the $\p D_1$.}
\end{figure}

\begin{example}[Imaging of multiple objects]\label{ex2}
Now, we use the imaging function $\mathfrak{F}(\ms,C)$ to identify multiple small objects $D_1$ and $D_2$. In this case, since SVD of $\mathbb{K}(C)$ is
\[\mathbb{K}(C)=\sum_{n=1}^{N}\tau_n\mU_n\mV_n^*\approx\sum_{n=1}^{2}\tau_n\mU_n\mV_n^*,\]
the imaging function of \eqref{ImagingFunction} becomes
\[\mathfrak{F}(\ms,C)=\abs{\sum_{n=1}^{2}\langle\mW(\ms),\mU_n\rangle\langle\mW(\ms),\overline{\mV}_n\rangle}.\]

Figure \ref{SV2} shows the distribution of singular values of $\mathbb{K}(C)$ with various $C$. Here, it is possible to discriminate nonzero singular values $\tau_1$ and $\tau_2$ when $C=0$ and $C=0.001i$. However, if the value of $|C|$ is not close to zero, it is difficult to discriminate two nonzero singular values. To define the imaging function, we selected $\tau_1$ as the nonzero singular value of $\mathbb{K}(C)$ if $C=0.01$, $C=0.01i$, and $0.1i$. For $C=0.1$, we selected the first six $\tau_n$ as nonzero singular values.

Figure \ref{Result2} shows the maps of $\mathfrak{F}(\ms,C)$ with various $C$. Similar to the single object imaging, the locations and shapes of $D_1$ and $D_2$ can be identified accurately through the map of $\mathfrak{F}(\ms,C)$ if $C=0$, $0001i$, and $0.01i$. It is possible to recognize $D_1$ and $D_2$, but several artifacts with small magnitudes are included in the imaging result if $C=0.01i$ and $C=0.1i$. Unfortunately, it is difficult to distinguish $D_1\cup D_2$ and several artifacts with large magnitudes if $C=0.1$. Therefore, similar to the single object imaging, the imaging result of $\mathfrak{F}(\ms,C)$ significantly depends on the choice of $C$. Additionally, good results can be obtained when $C=0$ or $|C|$ is sufficiently small for identifying multiple objects.
\end{example}

\begin{figure}[h]
\subfigure[$C=0$]{\includegraphics[width=0.33\textwidth]{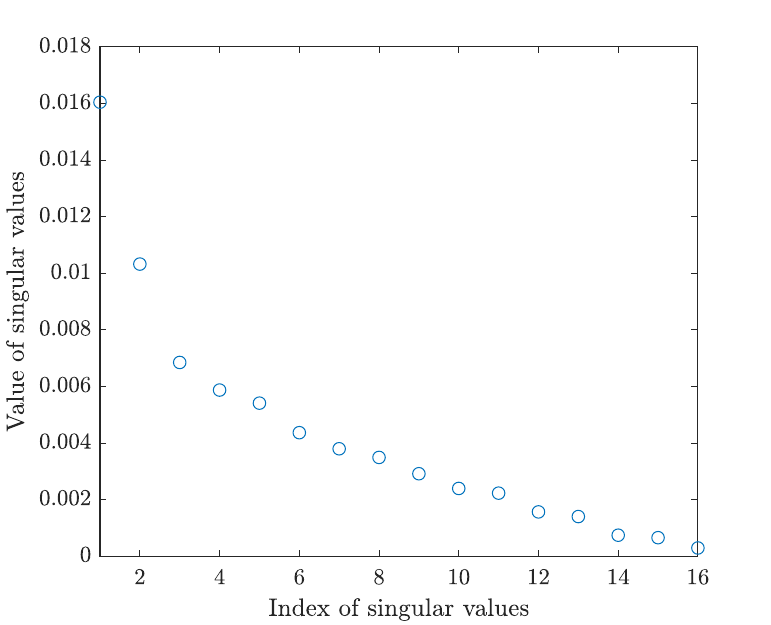}}\hfill
\subfigure[$C=0.01$]{\includegraphics[width=0.33\textwidth]{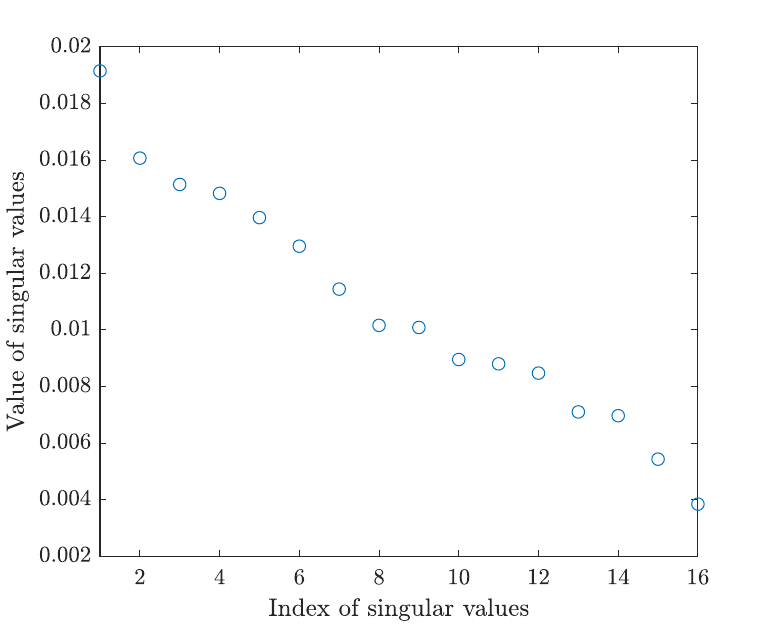}}\hfill
\subfigure[$C=0.1$]{\includegraphics[width=0.33\textwidth]{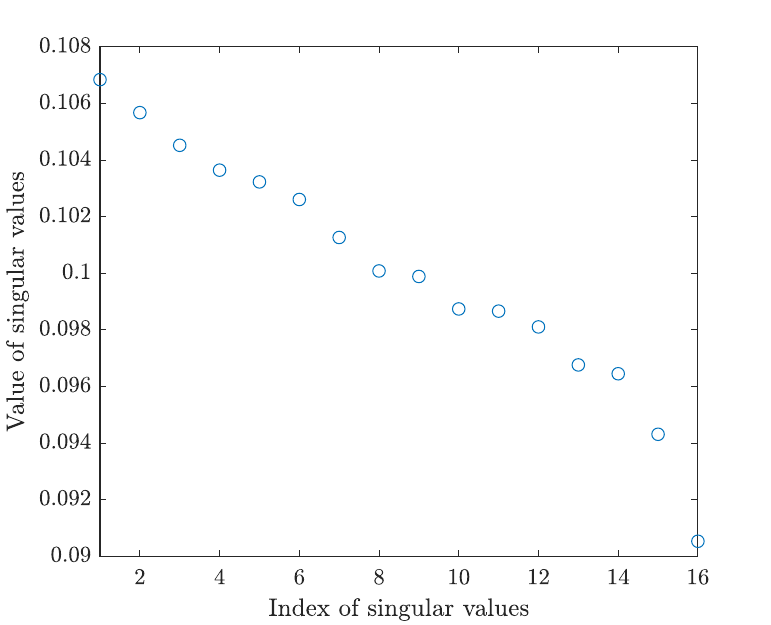}}\\
\subfigure[$C=0.001i$]{\includegraphics[width=0.33\textwidth]{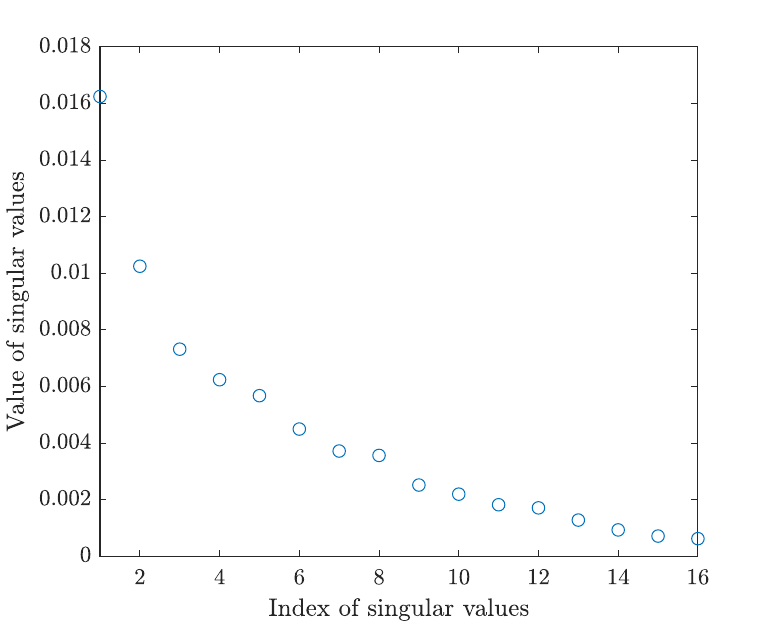}}\hfill
\subfigure[$C=0.01i$]{\includegraphics[width=0.33\textwidth]{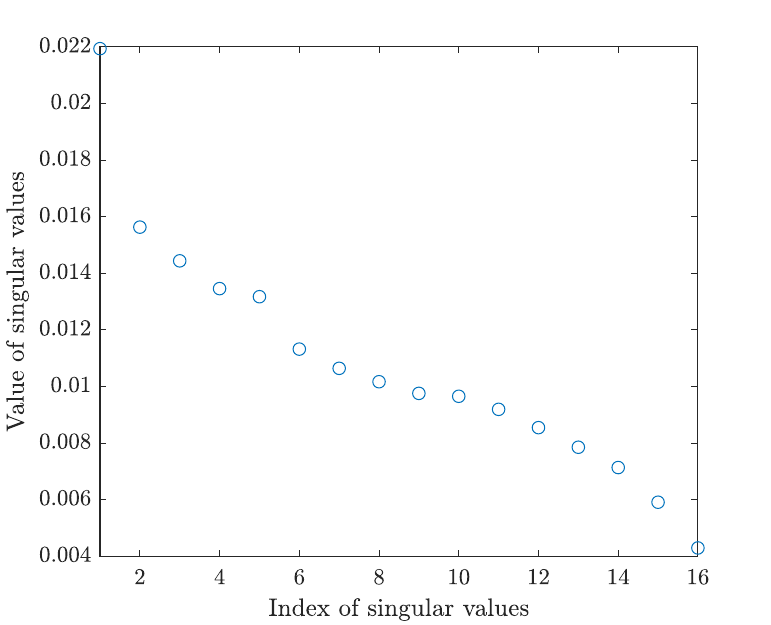}}\hfill
\subfigure[$C=0.1i$]{\includegraphics[width=0.33\textwidth]{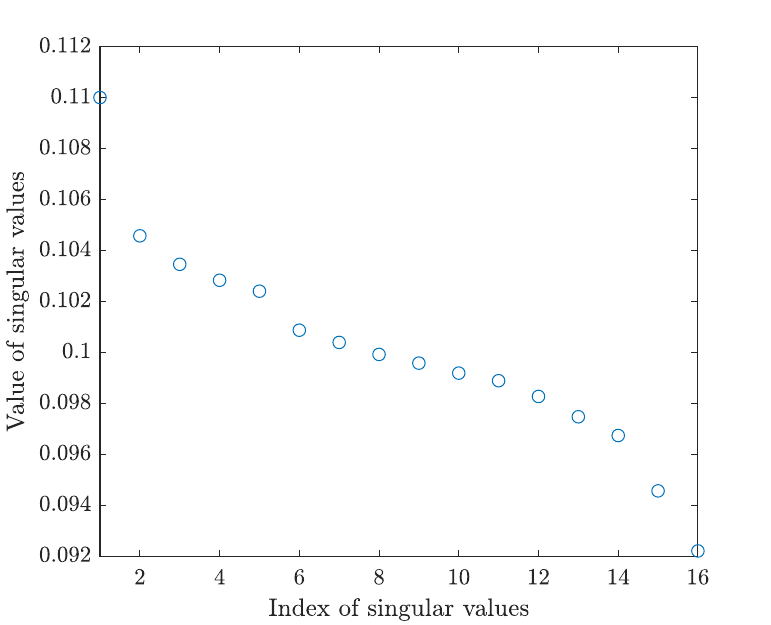}}
\caption{\label{SV2}(Example \ref{ex2}) Distribution of the singular values of $\mathbb{K}(C)$ at $f=\SI{1}{\giga\hertz}$ with various $C$.}
\end{figure}

\begin{figure}[h]
\subfigure[$C=0$]{\includegraphics[width=0.33\textwidth]{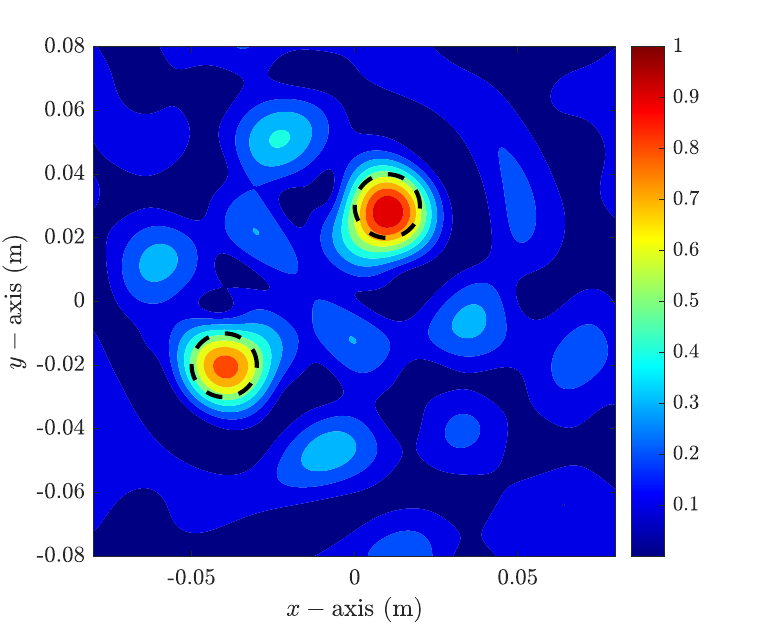}}\hfill
\subfigure[$C=0.01$]{\includegraphics[width=0.33\textwidth]{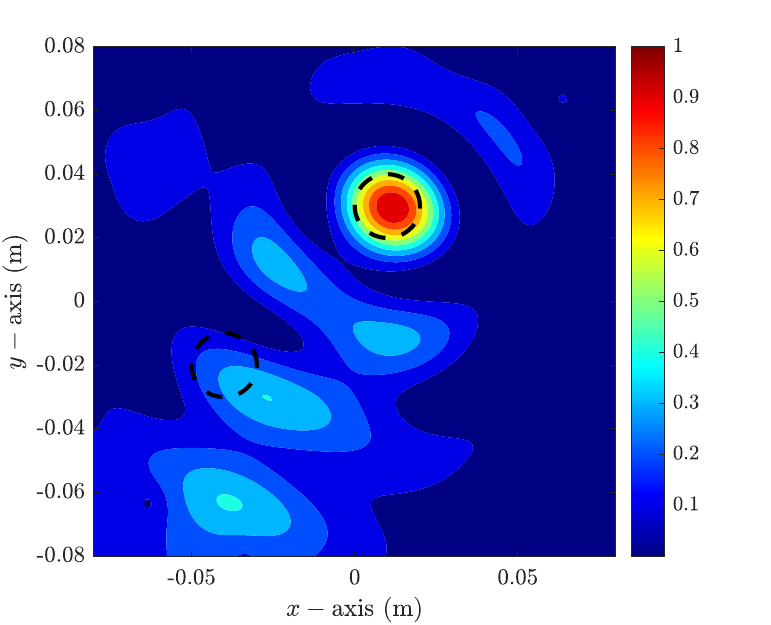}}\hfill
\subfigure[$C=0.1$]{\includegraphics[width=0.33\textwidth]{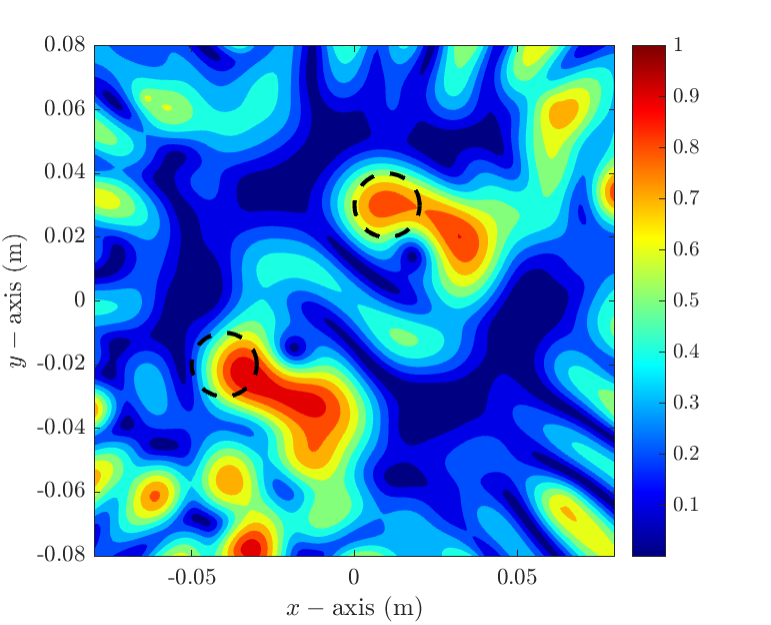}}\hfill\\
\subfigure[$C=0.001i$]{\includegraphics[width=0.33\textwidth]{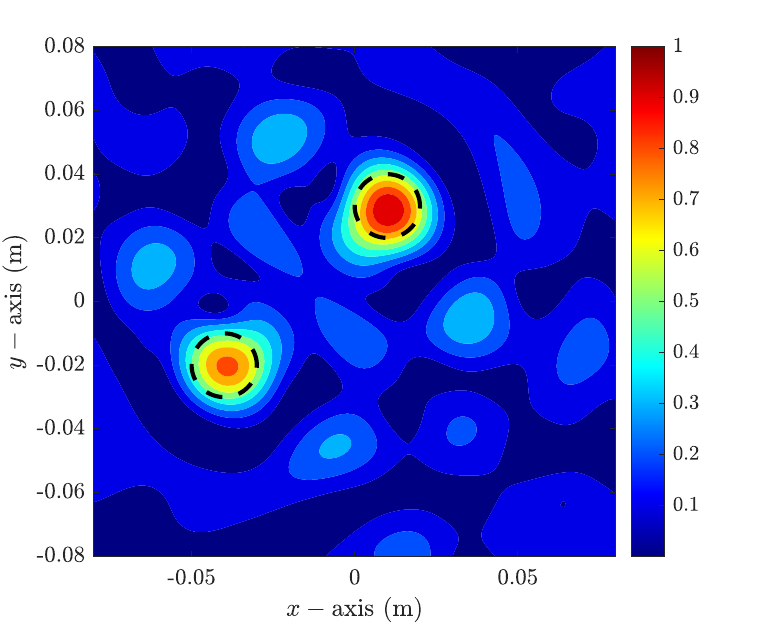}}\hfill
\subfigure[$C=0.01i$]{\includegraphics[width=0.33\textwidth]{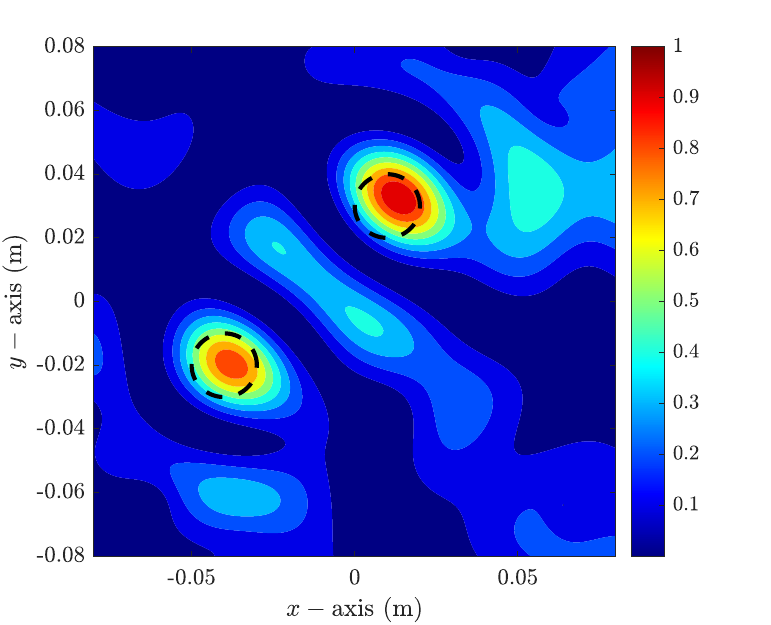}}\hfill
\subfigure[$C=0.1i$]{\includegraphics[width=0.33\textwidth]{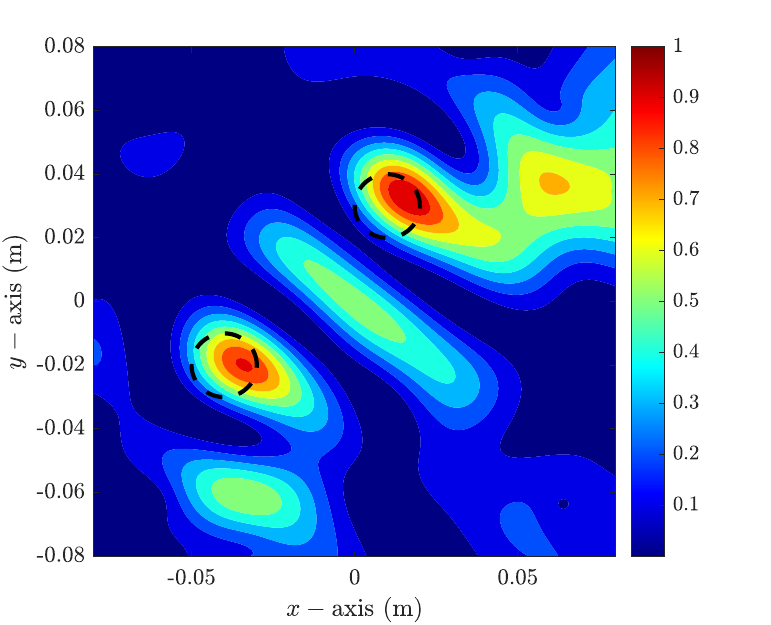}}\hfill
\caption{\label{Result2}(Example \ref{ex2}) Maps of $\mathfrak{F}(\ms,C)$ at $f=\SI{1}{\giga\hertz}$ with various $C$. Black-colored dashed lines describe the $\p D_1\cup\p D_2$.}
\end{figure}

\begin{example}[Imaging of large object]\label{ex3}
Although the equation \eqref{Approximation1} holds for small object, we now apply the $\mathfrak{F}(\ms,C)$ for retrieving a large object. Figures \ref{SV3} and \eqref{Result3} exhibit the distribution of singular values of $\mathbb{K}(C)$ and corresponding imaging results of $\mathbb{K}(C)$ in the presence of a circular object with radius $\alpha=\SI{0.048}{\meter}$ and material properties $(\eps,\sigma)=(15\eps_0,\SI{0.5}{\siemens/\meter})$. Similarly with the previous results in Examples \ref{ex1} and \ref{ex2}, it is possible to recognize the outline shape of object if $C=0$ and $C=0.001i$, i.e., if the value of $|C|$ is very small. However, it is impossible to identify the object if $|C|$ is not small such that $|C|\geq0.01$.
\end{example}

\begin{figure}[h]
\subfigure[$C=0$]{\includegraphics[width=0.33\textwidth]{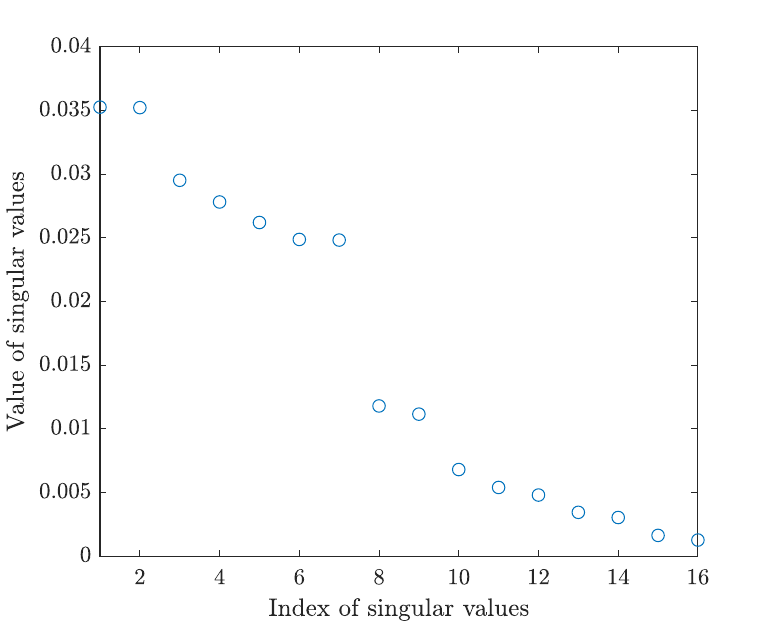}}\hfill
\subfigure[$C=0.01$]{\includegraphics[width=0.33\textwidth]{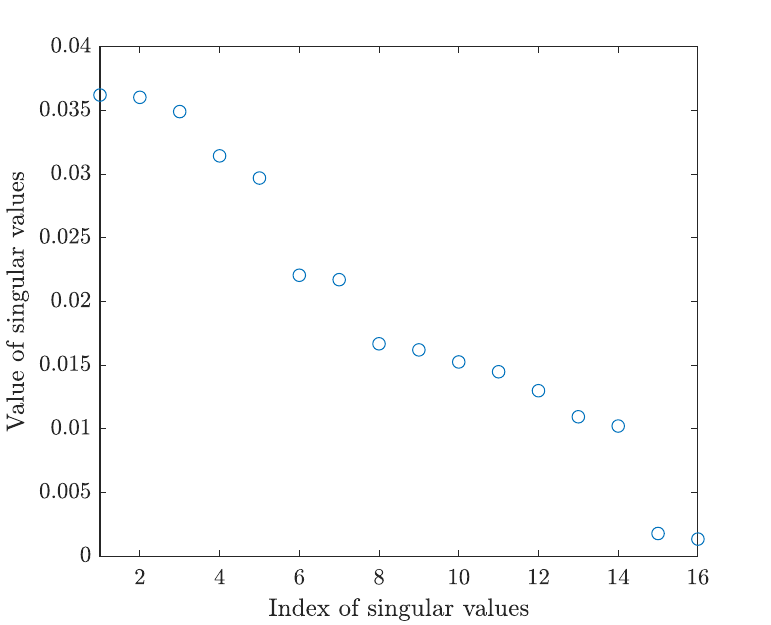}}\hfill
\subfigure[$C=0.1$]{\includegraphics[width=0.33\textwidth]{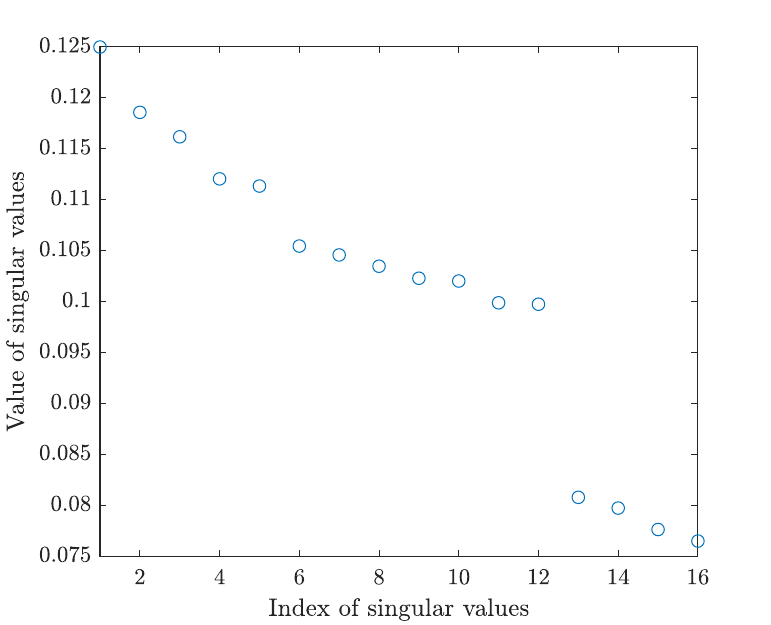}}\\
\subfigure[$C=0.001i$]{\includegraphics[width=0.33\textwidth]{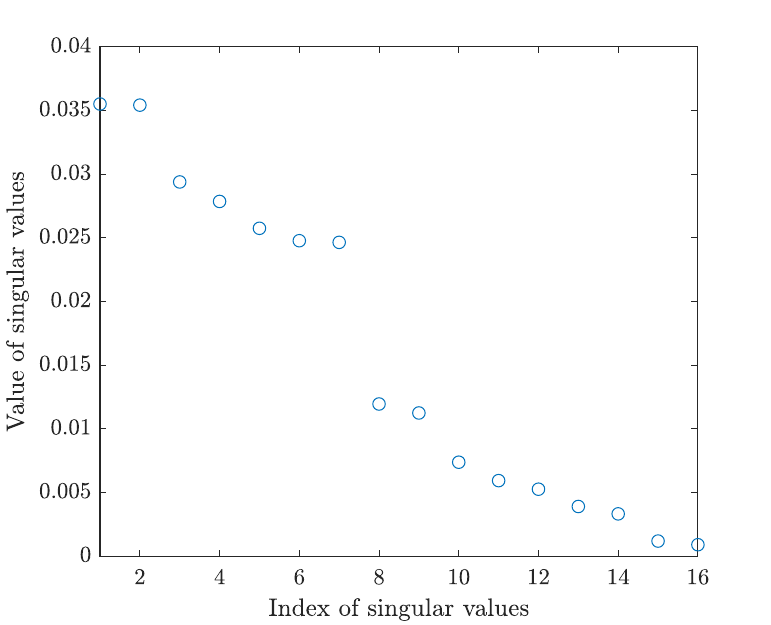}}\hfill
\subfigure[$C=0.01i$]{\includegraphics[width=0.33\textwidth]{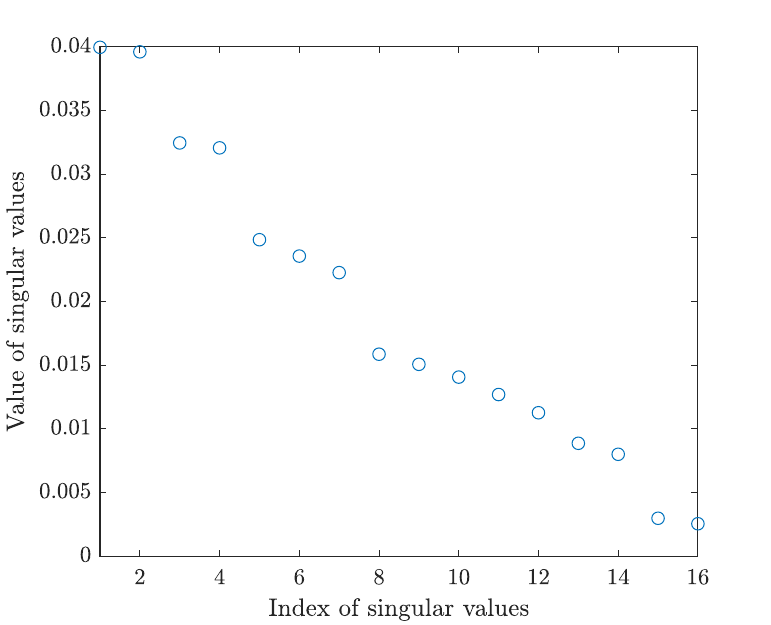}}\hfill
\subfigure[$C=0.1i$]{\includegraphics[width=0.33\textwidth]{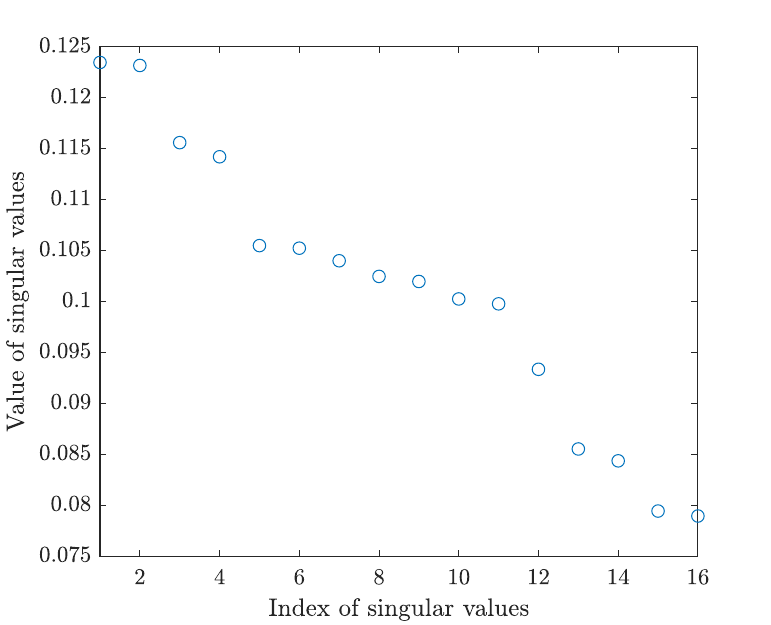}}
\caption{\label{SV3}(Example \ref{ex3}) Distribution of the singular values of $\mathbb{K}(C)$ at $f=\SI{1}{\giga\hertz}$ with various $C$.}
\end{figure}

\begin{figure}[h]
\subfigure[$C=0$]{\includegraphics[width=0.33\textwidth]{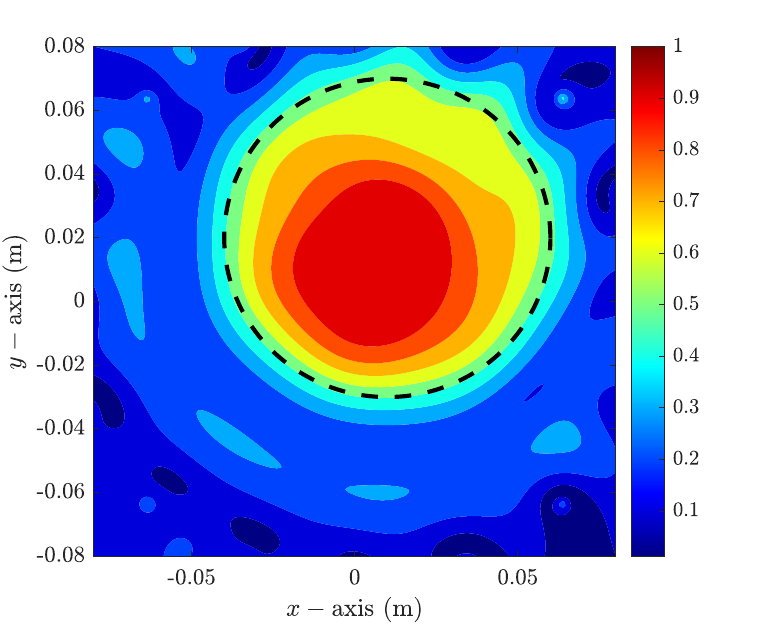}}\hfill
\subfigure[$C=0.01$]{\includegraphics[width=0.33\textwidth]{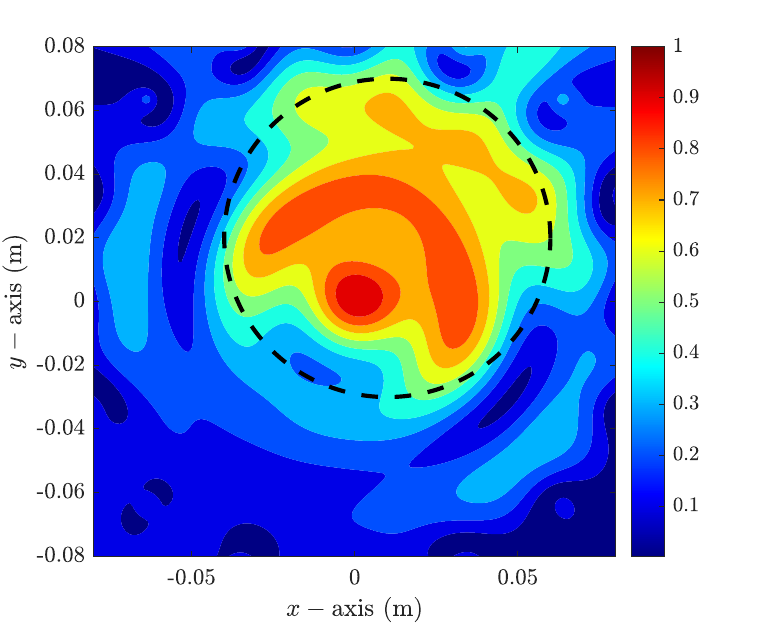}}\hfill
\subfigure[$C=0.1$]{\includegraphics[width=0.33\textwidth]{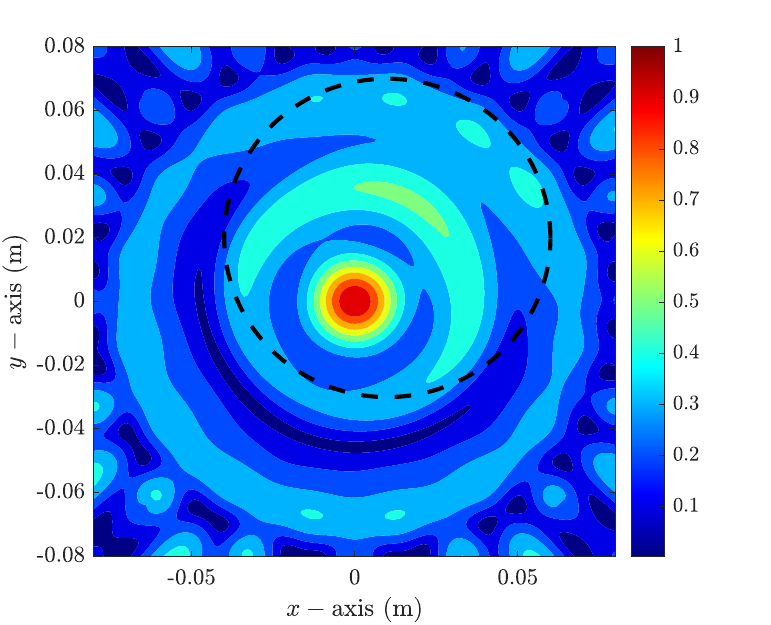}}\hfill\\
\subfigure[$C=0.001i$]{\includegraphics[width=0.33\textwidth]{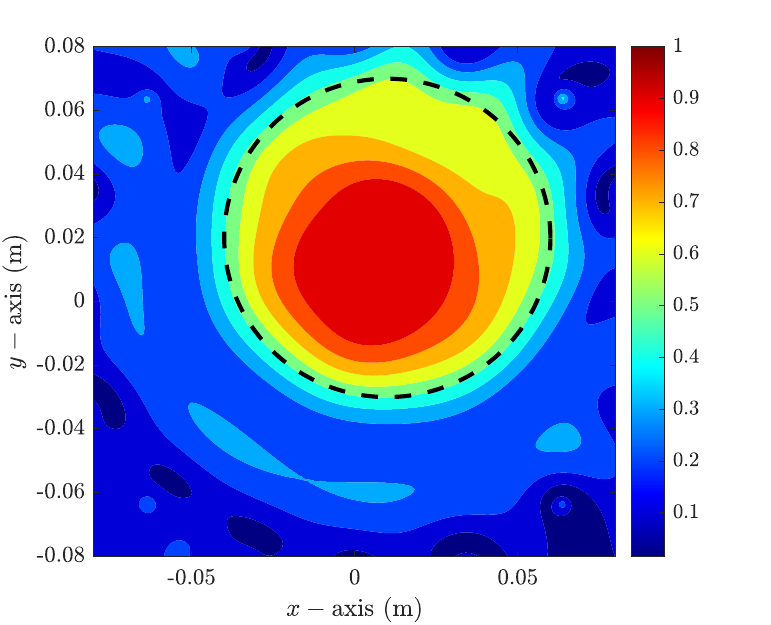}}\hfill
\subfigure[$C=0.01i$]{\includegraphics[width=0.33\textwidth]{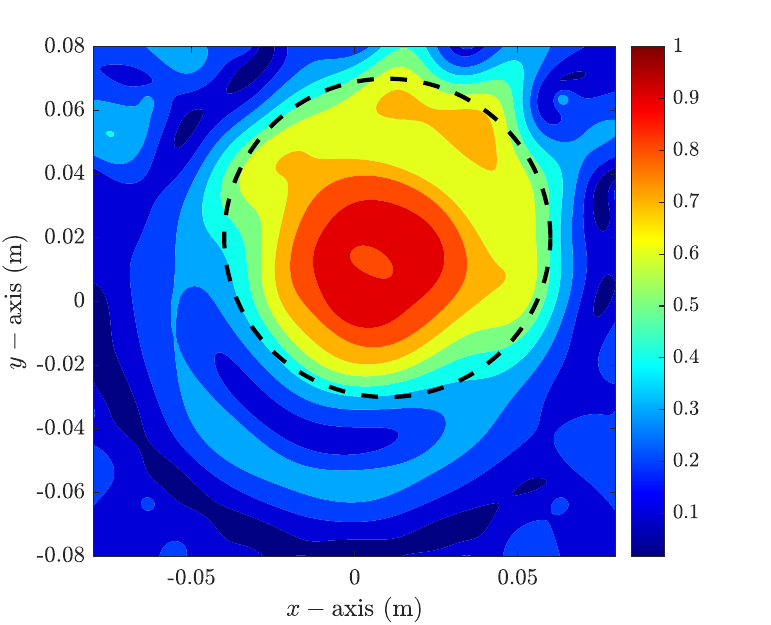}}\hfill
\subfigure[$C=0.1i$]{\includegraphics[width=0.33\textwidth]{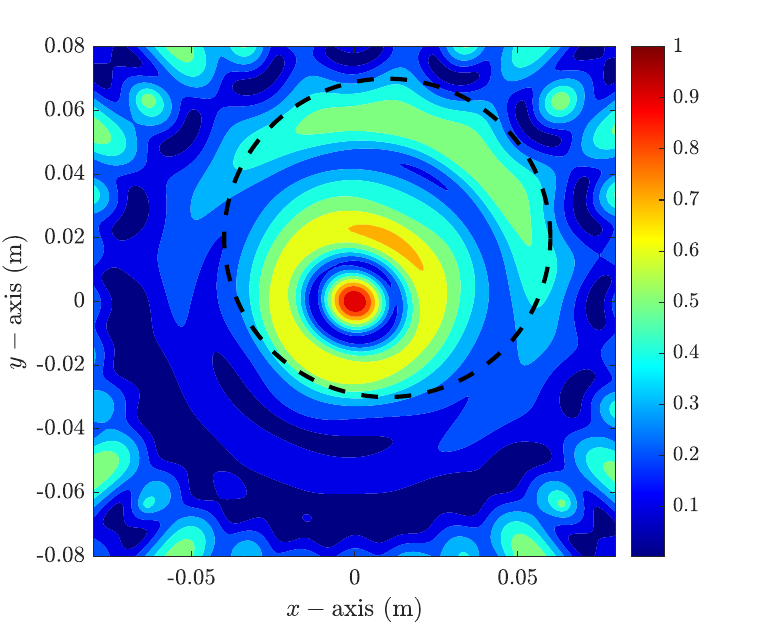}}\hfill
\caption{\label{Result3}(Example \ref{ex3}) Maps of $\mathfrak{F}(\ms,C)$ at $f=\SI{1}{\giga\hertz}$ with various $C$. Black-colored dashed line describes the $\p D$.}
\end{figure}

\subsection{Simulation Results with Experimental Data}
We now apply the imaging function of \eqref{ImagingFunction} for identifying small objects and demonstrating the effect of converted constant $C$. To this end, the scattering matrix was constructed by using the microwave machine, which was filled by water, manufactured by the research group of the Electronics and Telecommunications Research Institute (ETRI), refer to \cite{KLKJS}. The ROI $\Omega$ was selected as the interior of a circular region centered at the origin with radius $\SI{0.085}{\meter}$ and material properties $(\epsb,\sigmab)=(78\eps_0,\SI{0.2}{\siemens/\meter})$ at $f=\SI{925}{\mega\hertz}$. For an illustration of the microwave machine and selected objects (cross-section of screw drivers and plastic straw), we refer to Figure \ref{ConfigurationReal}.

\begin{figure}[h]
\subfigure[three screw drivers]{\includegraphics[width=0.5\textwidth]{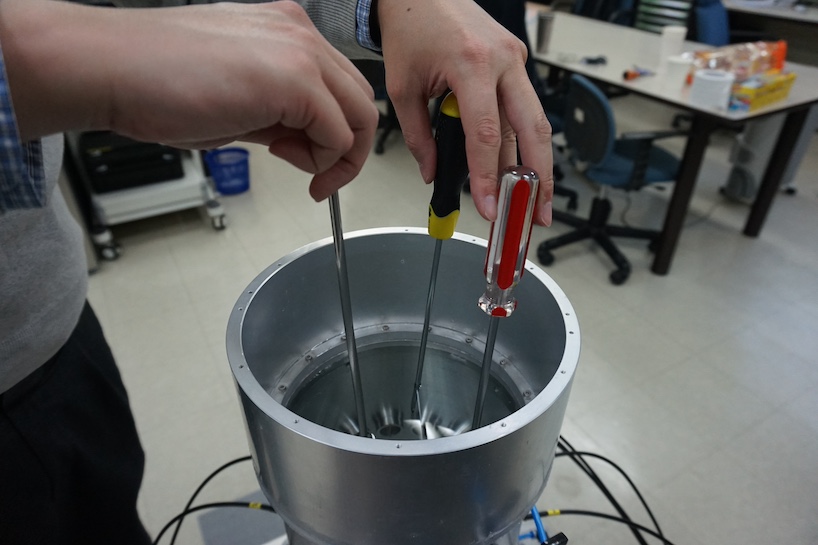}}\hfill
\subfigure[screw driver and plastic straw]{\includegraphics[width=0.5\textwidth]{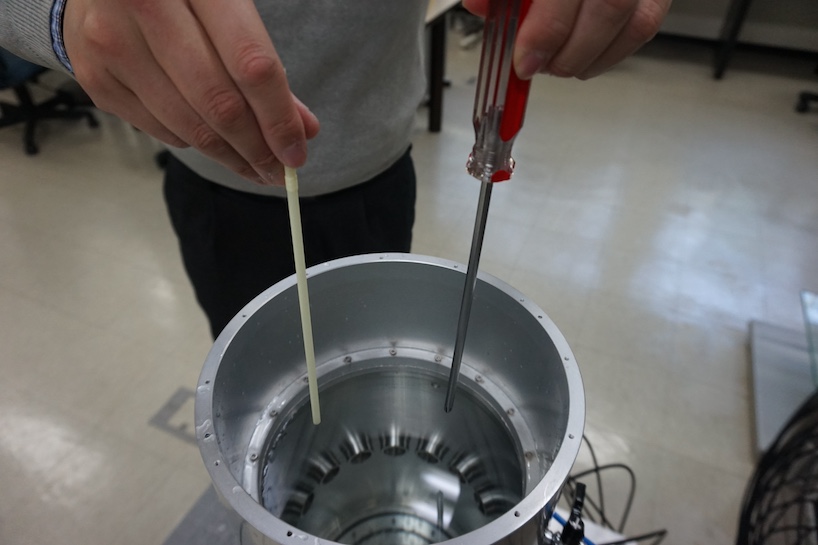}}
\caption{\label{ConfigurationReal}Pictures of real-data experiments for Examples \ref{exR1} (left) and \ref{exR2} (right).}
\end{figure}

\begin{example}[Imaging of small objects with same material properties]\label{exR1}
Here, we consider the imaging of cross-section of screw drivers $D_s$, $s=1,2,3$. Figure \ref{SVR1} shows the distribution of singular values of $\mathbb{K}(C)$ with various $C$. Similarly with the synthetic data experiment  (Example \ref{ex1}), it is easy to select first three singular values to define the imaging function when $C=0$, $0.01$, $0.001i$, and $0.01i$. However, it is not easy to select nonzero singular values when $C=0.1$ and $C=0.1i$.

Figure \ref{ResultR1} shows the maps of $\mathfrak{F}(\ms,C)$ with various $C$. Similar to the results in Example \ref{ex1}, the location and outline shape of objects can be identified very accurately through the map of $\mathfrak{F}(\ms,C)$ but some artifacts are included in the map when $|C|=0.01$. Unfortunately, it is very difficult to identify some objects when $C=0.1$ and distinguish the objects and artifacts with large magnitude when $C=0.1i$. Thus, same as the synthetic data experiments, the imaging result of $\mathfrak{F}(\ms,C)$ significantly depends on the choice of $C$ and good results can be obtained when $|C|$ is sufficiently small for identifying small objects.
\end{example}

\begin{figure}[h]
\subfigure[$C=0$]{\includegraphics[width=0.33\textwidth]{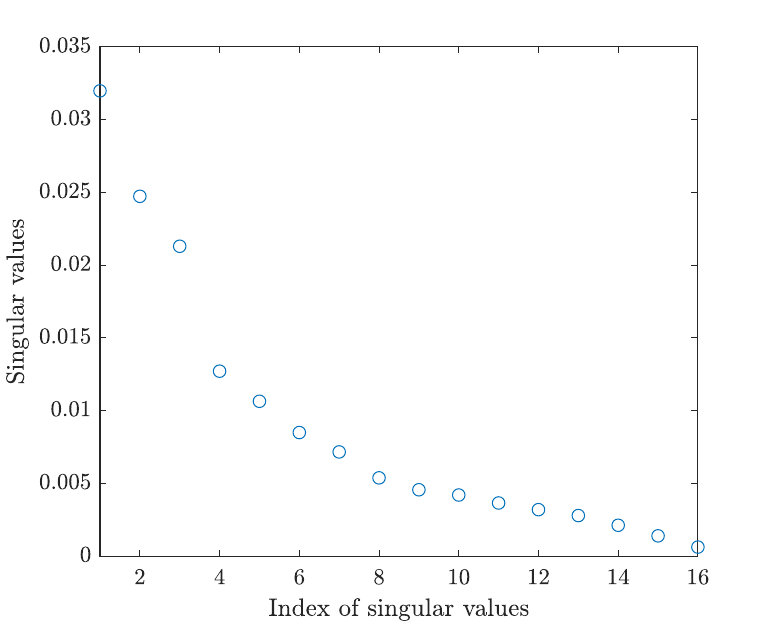}}\hfill
\subfigure[$C=0.01$]{\includegraphics[width=0.33\textwidth]{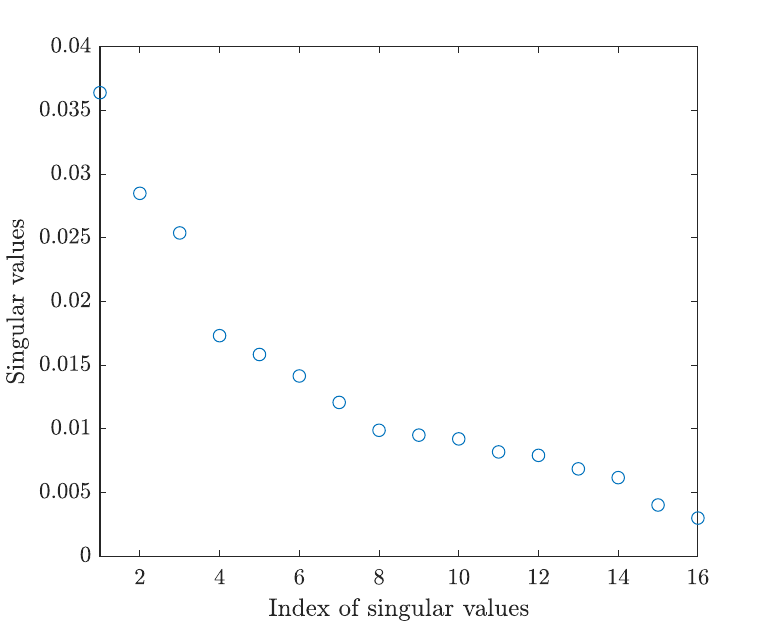}}\hfill
\subfigure[$C=0.1$]{\includegraphics[width=0.33\textwidth]{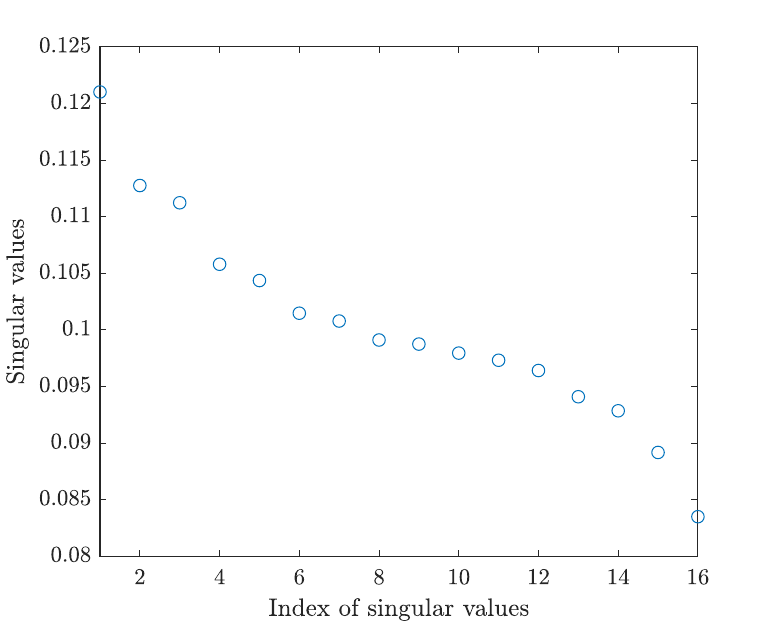}}\\
\subfigure[$C=0.001i$]{\includegraphics[width=0.33\textwidth]{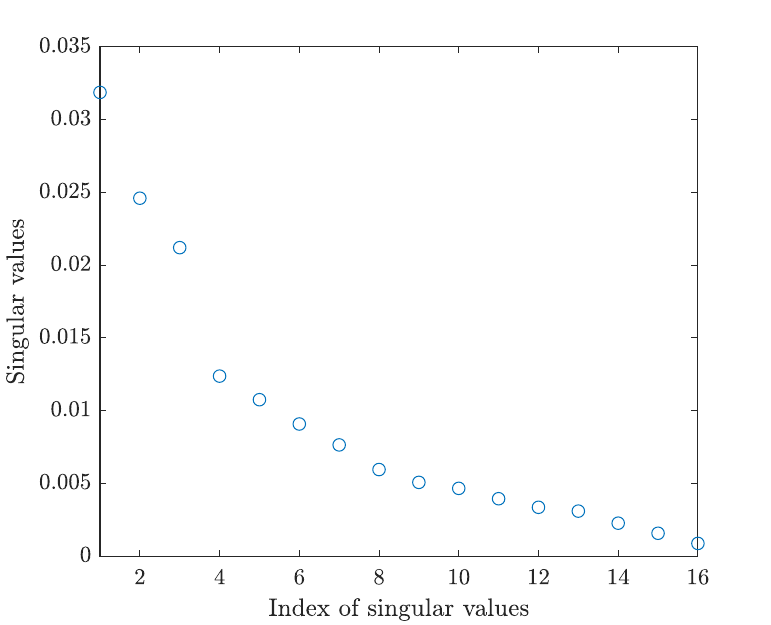}}\hfill
\subfigure[$C=0.01i$]{\includegraphics[width=0.33\textwidth]{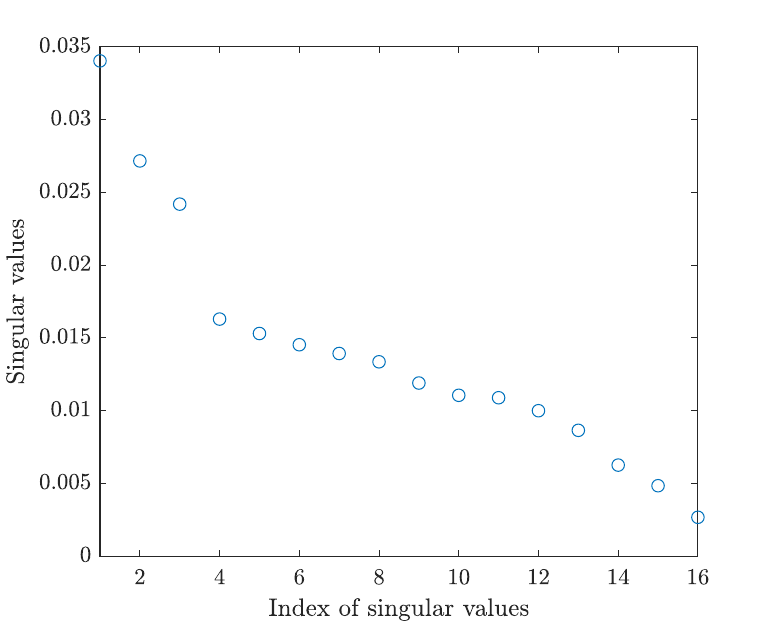}}\hfill
\subfigure[$C=0.1i$]{\includegraphics[width=0.33\textwidth]{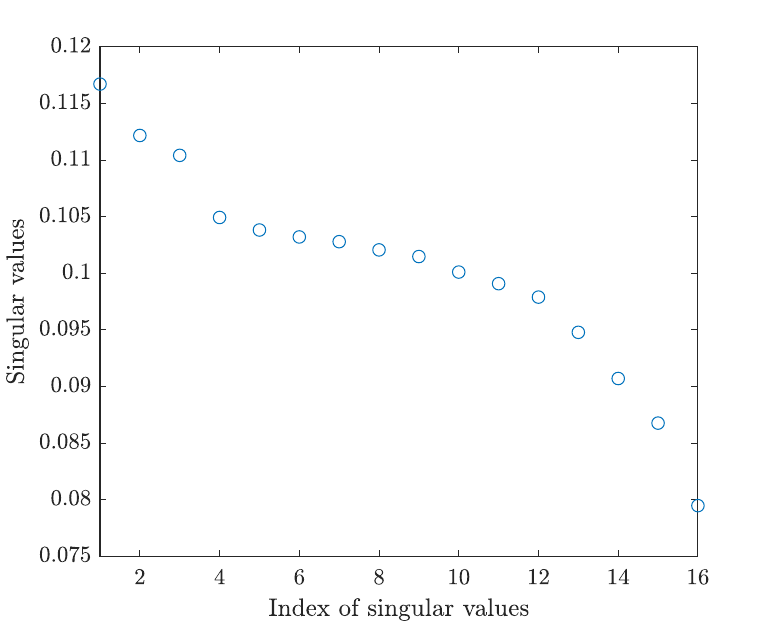}}
\caption{\label{SVR1}(Example \ref{exR1}) Distribution of the singular values of $\mathbb{K}(C)$ at $f=\SI{925}{\mega\hertz}$ with various $C$.}
\end{figure}

\begin{figure}[h]
\subfigure[$C=0$]{\includegraphics[width=0.33\textwidth]{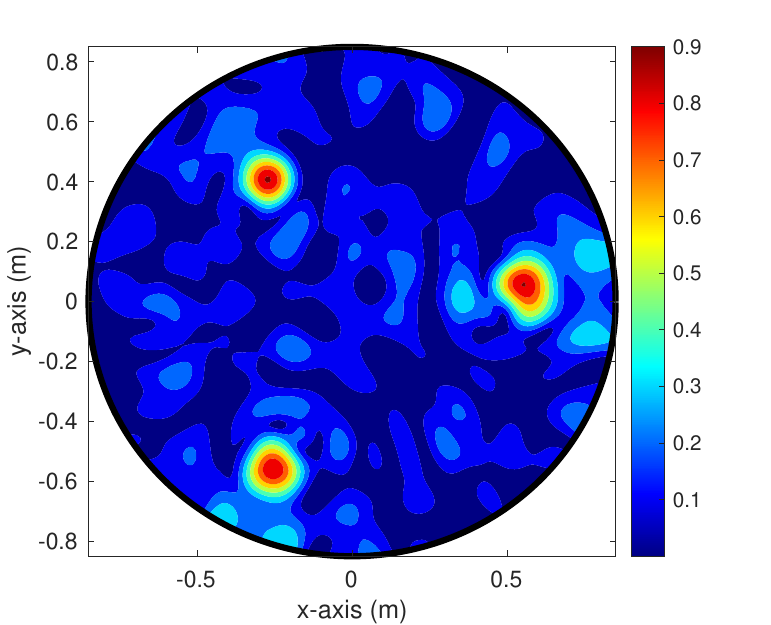}}\hfill
\subfigure[$C=0.01$]{\includegraphics[width=0.33\textwidth]{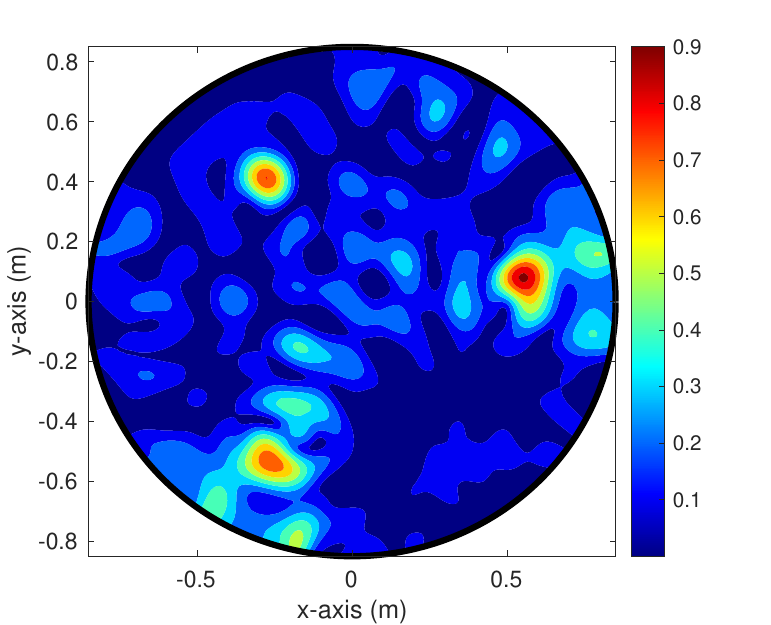}}\hfill
\subfigure[$C=0.1$]{\includegraphics[width=0.33\textwidth]{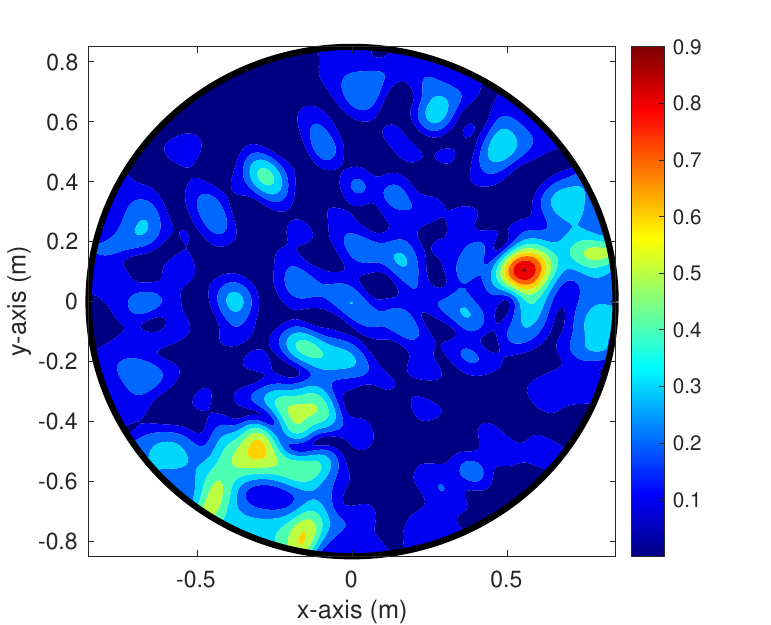}}\hfill\\
\subfigure[$C=0.001i$]{\includegraphics[width=0.33\textwidth]{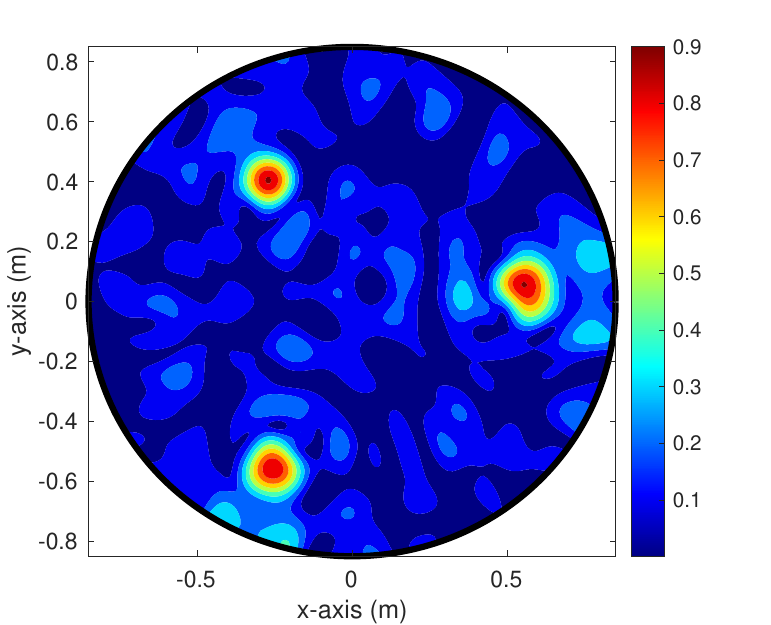}}\hfill
\subfigure[$C=0.01i$]{\includegraphics[width=0.33\textwidth]{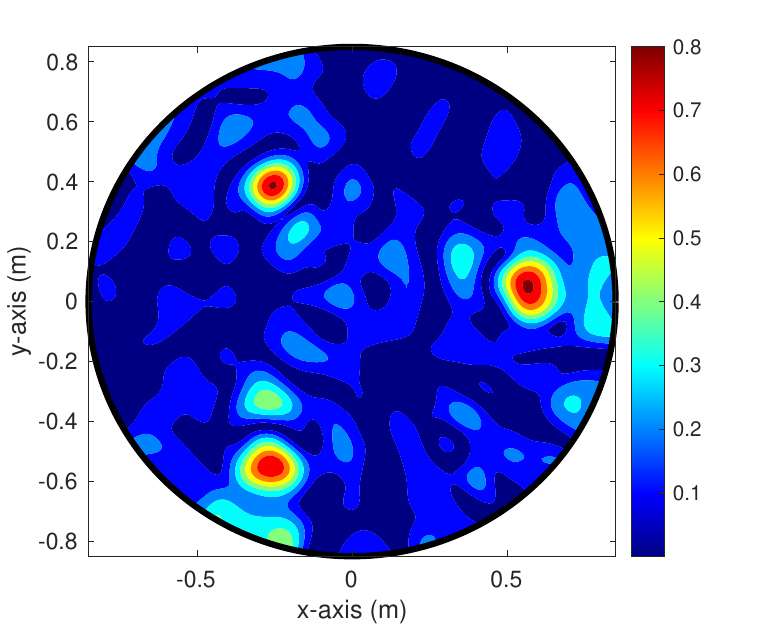}}\hfill
\subfigure[$C=0.1i$]{\includegraphics[width=0.33\textwidth]{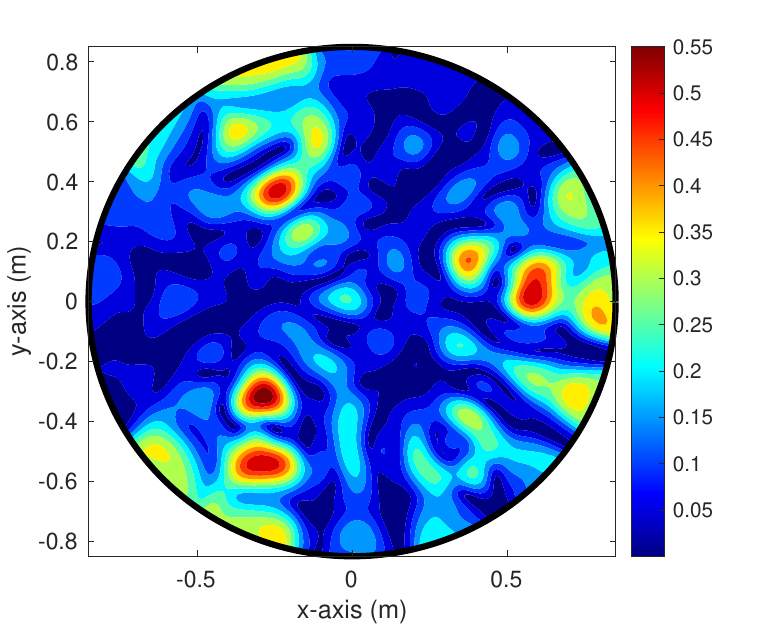}}\hfill
\caption{\label{ResultR1}(Example \ref{exR1}) Maps of $\mathfrak{F}(\ms,C)$ at $f=\SI{925}{\mega\hertz}$ with various $C$.}
\end{figure}

\begin{example}[Imaging of small objects with different material properties]\label{exR2}
Now, we exhibit a result for identifying cross-section of   screw driver and plastic straw. Similar to the previous result in Example \ref{exR1}, there is no difficult to select nonzero singular value when $C=0$, $0.01$, $0.001i$, and $0.01i$ but it is still difficult to select when $C=0.1$ and $C=0.1i$. In this case, we selected $\tau_1$ as the nonzero singular value of $\mathbb{K}(C)$ by regarding differences $\tau_n-\tau_{n+1}$, $n=1,2,\cdots,15$. It is worth to emphasize that the values of permittivity of the plastic straw and screw driver are extremely small and large, respectively. Thus, the existence of plastic straw does not affect to the $\mathcal{S}_{\meas}(n,m)$ because the value of $\mathcal{O}_D$ of \eqref{ObjectIncident} for screw driver is extremely larger than the one of plastic straw. Correspondingly, only one singular value of $\mathbb{K}(C)$ that is significantly larger than the others has been appeared.

Figure \ref{ResultR2} shows the maps of $\mathfrak{F}(\ms,C)$ with various $C$. Opposite to the previous results, the location and outline shape of cross-section of screw driver can be identified very accurately through the map of $\mathfrak{F}(\ms,C)$ with $|C|\leq0.1$. Unfortunately, it is impossible to recognize the existence of the cross-section of screw driver for any value $C$.
\end{example}

\begin{figure}[h]
\subfigure[$C=0$]{\includegraphics[width=0.33\textwidth]{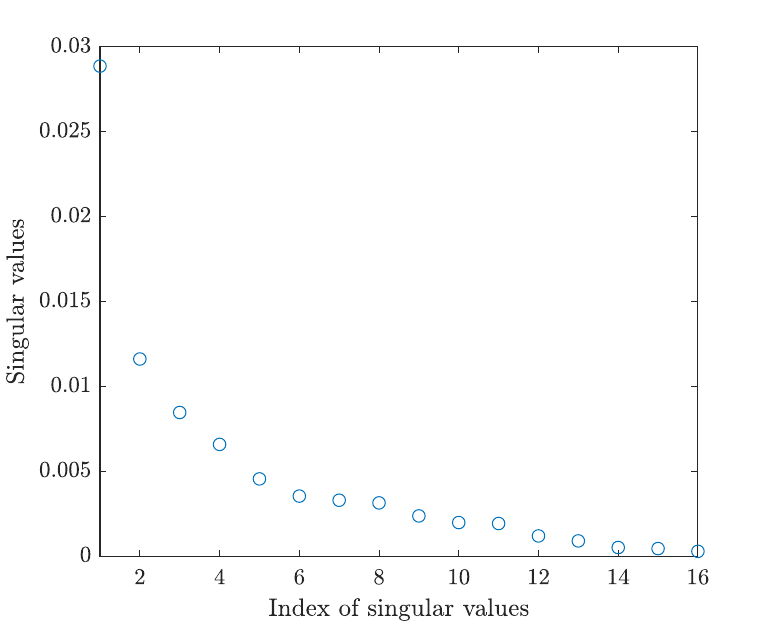}}\hfill
\subfigure[$C=0.01$]{\includegraphics[width=0.33\textwidth]{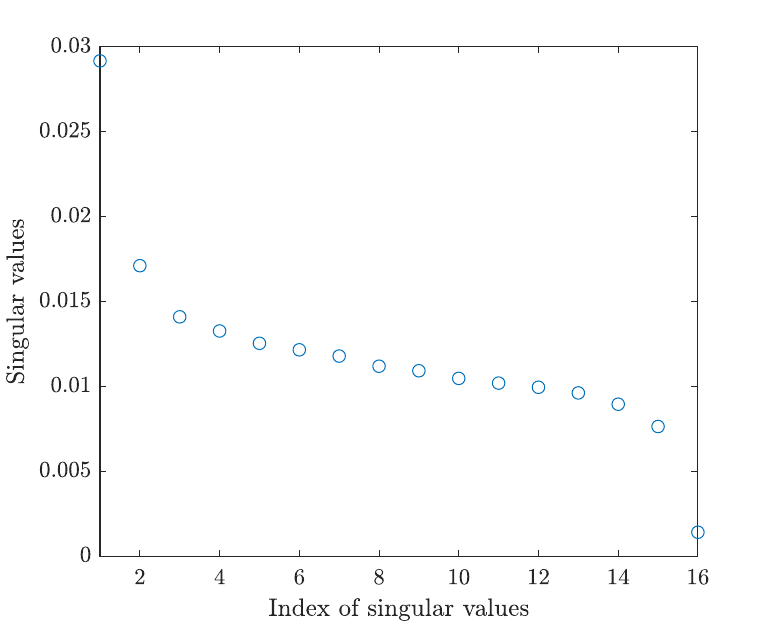}}\hfill
\subfigure[$C=0.1$]{\includegraphics[width=0.33\textwidth]{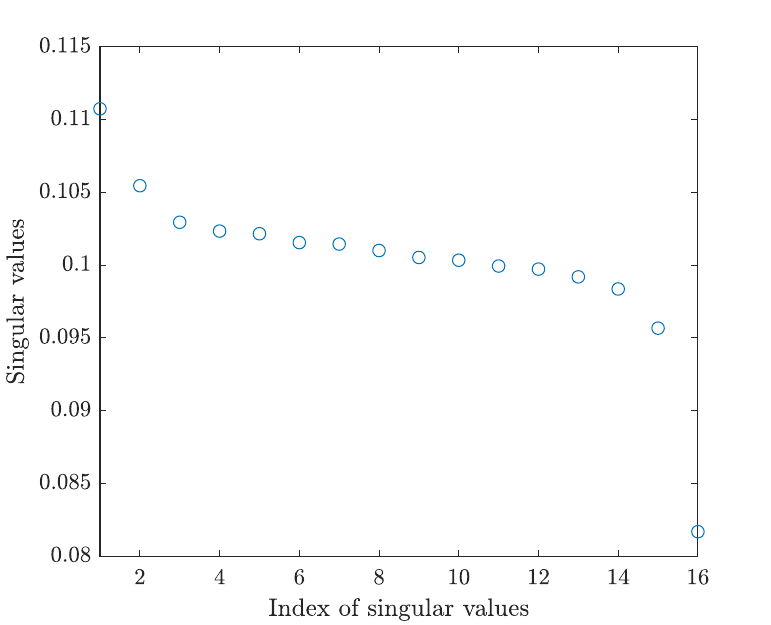}}\\
\subfigure[$C=0.001i$]{\includegraphics[width=0.33\textwidth]{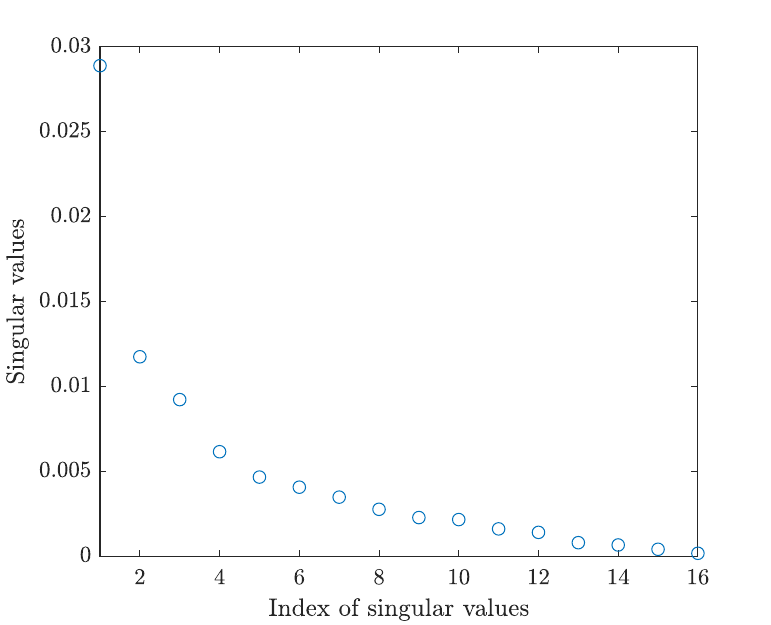}}\hfill
\subfigure[$C=0.01i$]{\includegraphics[width=0.33\textwidth]{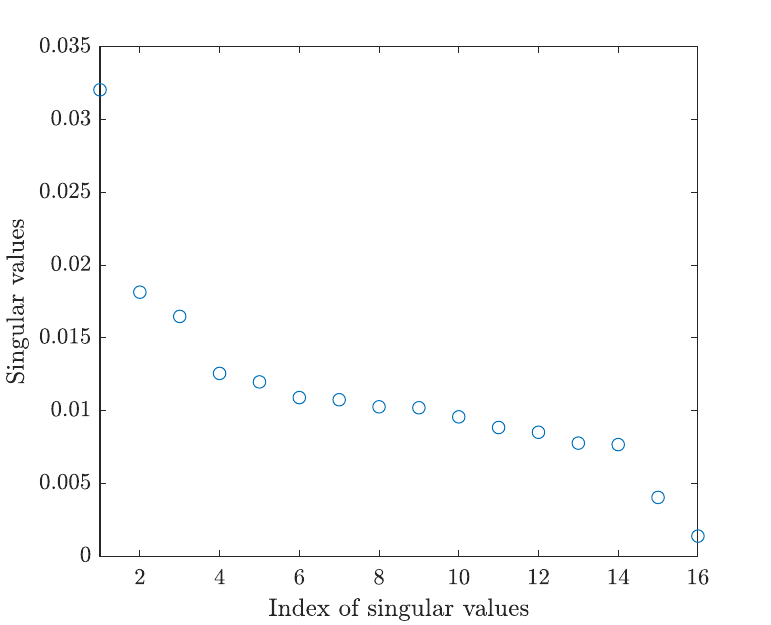}}\hfill
\subfigure[$C=0.1i$]{\includegraphics[width=0.33\textwidth]{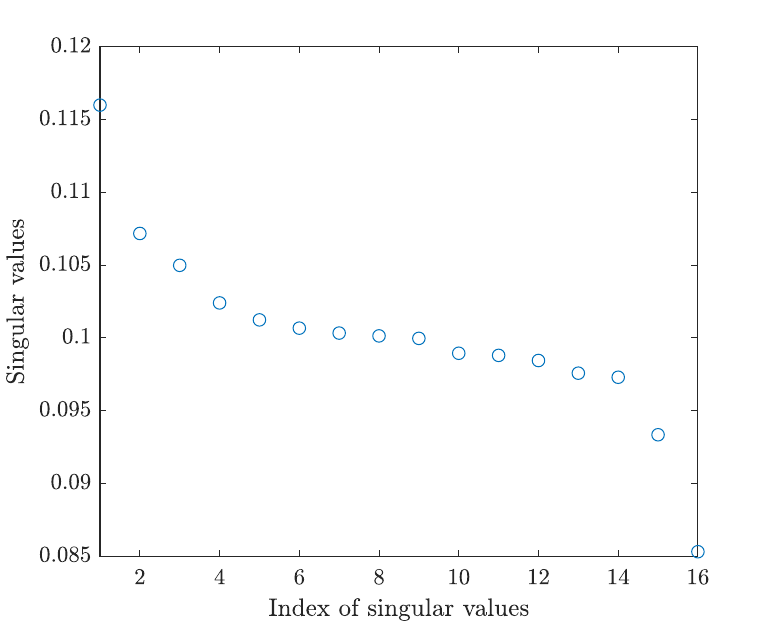}}
\caption{\label{SVR2}(Example \ref{exR2}) Distribution of the singular values of $\mathbb{K}(C)$ at $f=\SI{925}{\mega\hertz}$ with various $C$.}
\end{figure}

\begin{figure}[h]
\subfigure[$C=0$]{\includegraphics[width=0.33\textwidth]{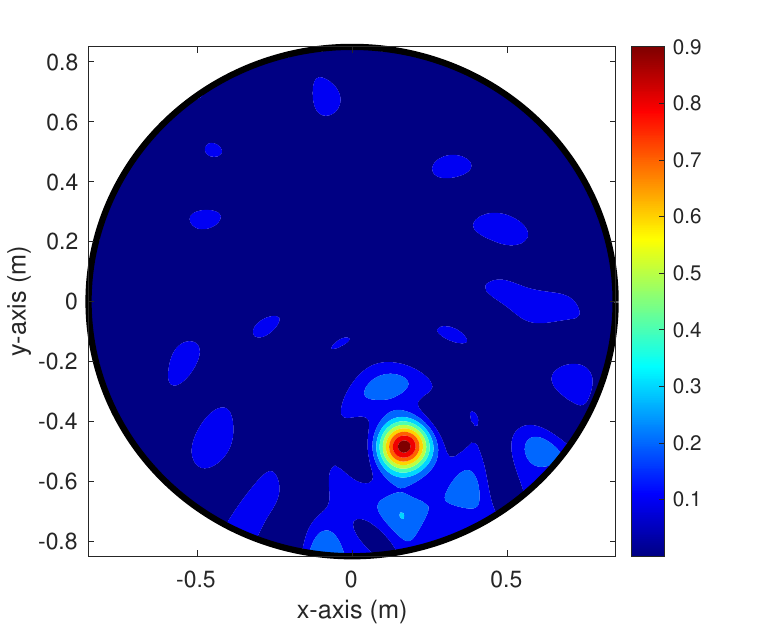}}\hfill
\subfigure[$C=0.01$]{\includegraphics[width=0.33\textwidth]{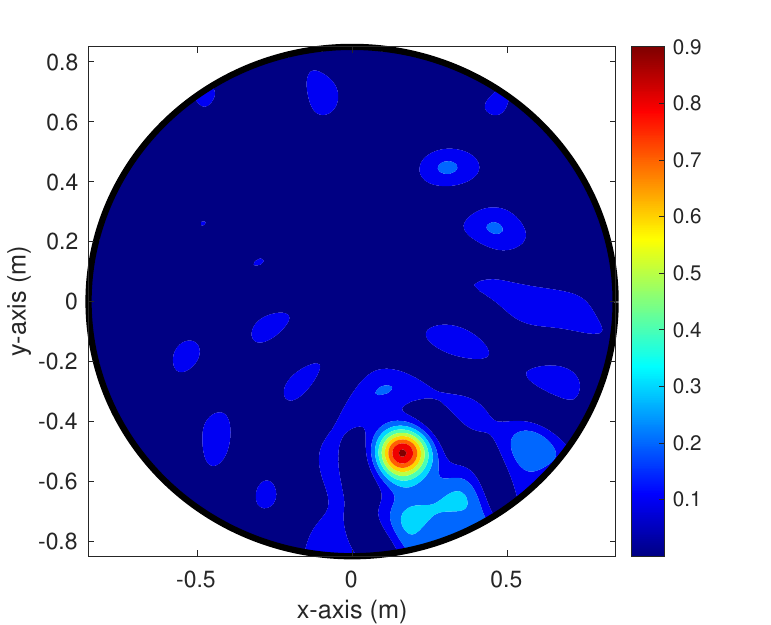}}\hfill
\subfigure[$C=0.1$]{\includegraphics[width=0.33\textwidth]{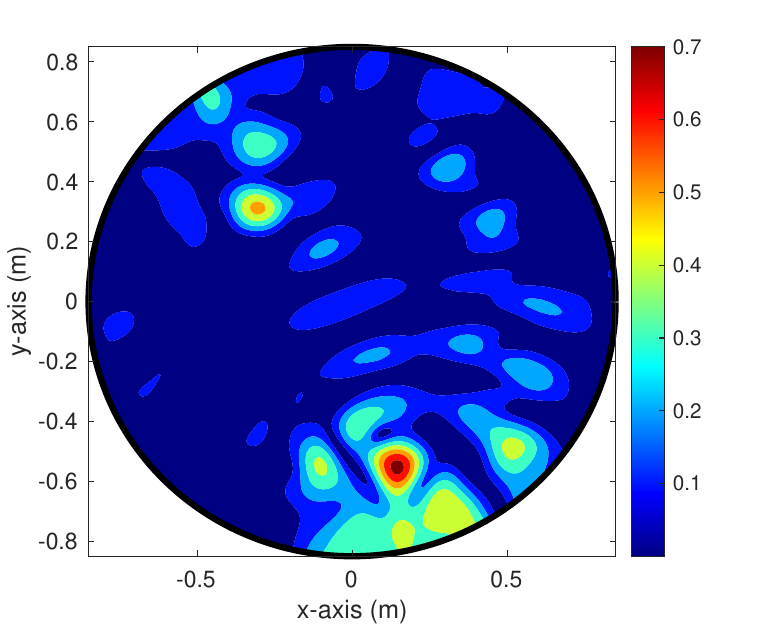}}\hfill\\
\subfigure[$C=0.001i$]{\includegraphics[width=0.33\textwidth]{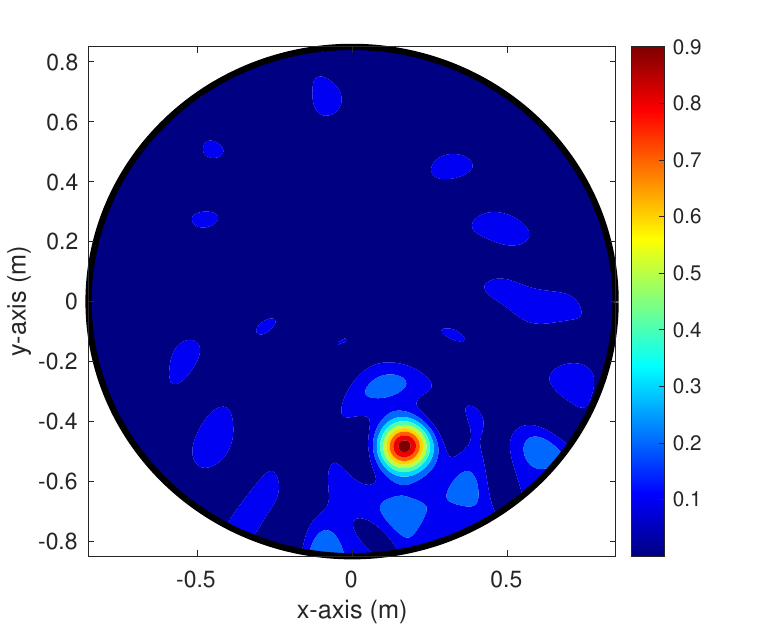}}\hfill
\subfigure[$C=0.01i$]{\includegraphics[width=0.33\textwidth]{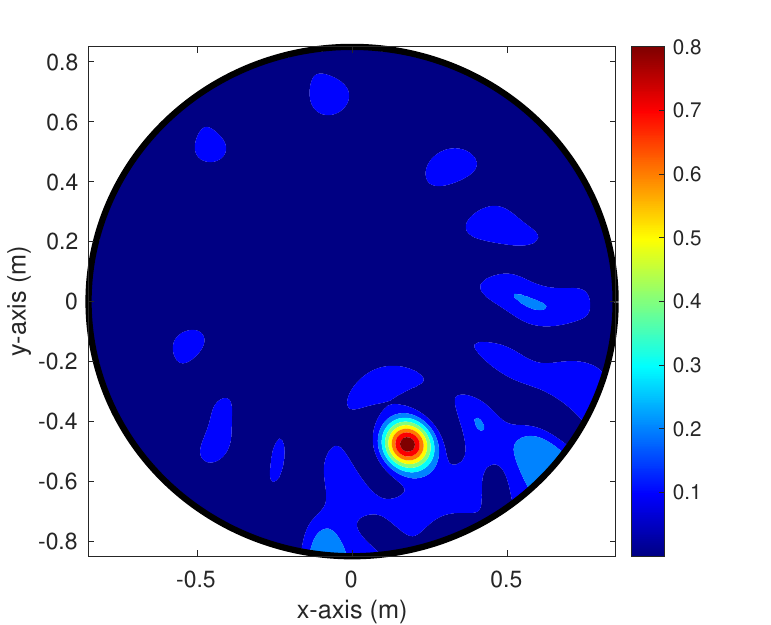}}\hfill
\subfigure[$C=0.1i$]{\includegraphics[width=0.33\textwidth]{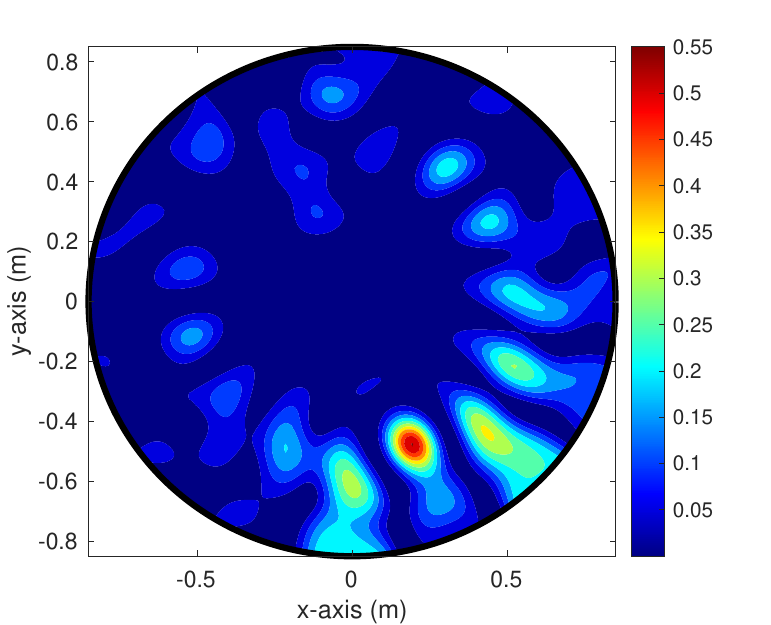}}\hfill
\caption{\label{ResultR2}(Example \ref{exR2}) Maps of $\mathfrak{F}(\ms,C)$ at $f=\SI{925}{\mega\hertz}$ with various $C$.}
\end{figure}

On the basis of the Theorem \ref{TheoremStructure}, Remark \ref{Remark2}, and simulation results with synthetic and experimental data, we can conclude that applying small value of $|C|$ guarantees good imaging results but the criteria for \textit{small value} is ambiguous. Hence, converting unknown diagonal data to the $C=0$ will be the best choice for a proper application of the SM. This is the theoretical reason why the zero constant was chosen instead of the unknown data in most studies.

\section{Conclusion}\label{sec:5}
In this study, we considered the application of SM for the quick identification of small objects by converting the unknown diagonal elements of the scattering matrix into a constant. We explained the effectiveness of this approach by demonstrating that the imaging function can be represented using an infinite series of Bessel functions of integer order, antenna number and arrangement, and applied constant. Moreover, we established that small objects can be uniquely retrieved when the absolute value of the applied constant is zero or close to zero.

It is important to note that our analysis and application scope in this paper is limited to retrieving small objects in a two-dimensional setting. However, future research focusing on identifying large objects and extending to three-dimensional microwave imaging would present interesting research topics.

\bibliographystyle{plain}
\bibliography{../../../References}

\end{document}